\documentclass{article}
\usepackage{url}

\author{I. Cravero, M. Semplice
\thanks{Correspondence to: Matteo Semplice -
Dipartimento di Matematica ``G. Peano'' -
Universit\`a di Torino -
Via C. Alberto, 10 -
10123 Torino (Italy). 
{matteo.semplice@unito.it}
}}

\date{}
\title{On the accuracy of WENO and CWENO reconstructions of third order on nonuniform meshes}

\usepackage{amsmath,amsthm}
\newtheorem{theorem}{Theorem}
\newtheorem{remark}[theorem]{Remark}
\newtheorem{lemma}[theorem]{Lemma}

\usepackage{graphicx}

\usepackage{xcolor}
\newcommand{\ca}[1]{\overline{#1}}
\newcommand{\D}{\mathrm{d}}
\newcommand{\bigO}{\mathcal{O}}
\newcommand{\up}{u^{\prime}_j}
\newcommand{\us}{u^{\prime\prime}_j}
\newcommand{\ut}{u^{\prime\prime\prime}_j}
\newcommand{\uq}{u^{(4)}_j}

\newcommand{\sn}[2]{\ensuremath{#1\cdot10^{#2}}}

\newcommand{\DT}{\mathrm{\Delta}t}

\begin{document}

\maketitle
\begin{abstract}
Third order WENO and CWENO reconstruction are widespread high order reconstruction techniques for numerical schemes for hyperbolic conservation and balance laws. In their definition, there appears a small positive parameter, usually called $\epsilon$,  that was originally introduced in order to avoid a division by zero on constant states, but whose value  was later shown to affect the convergence properties of the schemes. Recently, two detailed studies of the role of this parameter, in the case of uniform meshes, were published. In this paper we extend their results to the case of finite volume schemes on non-uniform meshes, which is very important for h-adaptive schemes, showing the benefits of choosing $\epsilon$ as a function of the local mesh size $h_j$. In particular we show that choosing $\epsilon=h_j^2$ or $\epsilon=h_j$ is beneficial for the error and convergence order, studying on several non-uniform grids the effect of this choice on the reconstruction error, on fully discrete schemes for the linear transport equation, on the stability of the numerical schemes. Finally we compare the different choices for $\epsilon$ in the case of a well-balanced scheme for the Saint-Venant system for shallow water flows and in the case of an h-adaptive scheme for nonlinear systems of conservation laws and show numerical tests for a two-dimensional generalisation of the CWENO reconstruction on locally adapted meshes.

\end{abstract}

\section{Introduction}\label{sec:intro}

Hyperbolic conservation and balance laws like
\begin{equation}\label{e:blaw}
u_t + \nabla \cdot f(u) = g(u,x)
\end{equation}
describe a wealth of phenomena, from gas dynamics, to magnetohydrodynamics, to shallow water flows, that are of much practical interest in fields like the study of industrial processes, simulation-based prototyping, forecasting of natural phenomena for the design of early-warning systems and passive defences against natural disasters, etc. Nonlinearities in the flux function $f(u)$ give rise to very complex solutions, whose efficient numerical approximation often requires Adaptive Mesh Refinement (AMR) techniques, so that one can use fine meshes only in specific regions of the flow, e.g. close to boundaries, shocks, contact discontinuities. 

In the field of hyperbolic conservation laws, AMR has been achieved by superimposing patches of uniform grids, thus representing and time-advancing the solution at the same location on multiple meshes, taking care to maintain conservativity at mesh interfaces \cite{BL98:amrclaw}. This approach undoubtedly has the advantage that one can employ the very well-known and well-studied numerical schemes for discretization on uniform meshes. Alternatively, one may consider a single mesh, that can be locally refined, giving rise to an unstructured mesh. In this respect, it is important to notice that finite-volume methods can be easily formulated also on nonconforming meshes, like quad-tree ones, that can present hanging nodes. This approach, albeit facilitating the grid management and enforcement of conservativity, requires the development of new discretization techniques, suitable for irregular meshes.

The fact that, in smooth regions of the flow, schemes of order three and above can compute accurate solutions on relatively coarse meshes is very important in AMR schemes, because a good adaptive strategy will be able to exploit this by leaving the mesh very coarse in smooth regions, and refining it only in a very small region around problematic features. For example, in \cite{CRS} a second and a third order adaptive scheme with the same error indicator are compared, showing that the second order one has to refine in a much larger region, giving rise to much bigger meshes.

Since a successful technique for mesh adaptivity in the case of hyperbolic systems is expected to continuously refine and coarsen the mesh at different locations in the computational domain, in our opinion it is important that discretization techniques do not to rely on precomputed quantities related to the local mesh geometry (see e.g. \cite{HuShu:1999}). A good balance between the high accuracy and the small stencil requirement are third order accurate schemes. In this work we thus concentrate on two third order reconstructions that employ very small stencils, for which the needed metric information on the neighbours may be gathered efficiently at every timestep and extend to the non-uniform case a technique that can squeeze out from these small stencils a lot of extra accuracy with respect to the traditional form of the discretization.

In order to fix the notation,  we consider systems of balance laws with a geometric source term of the form \eqref{e:blaw}
and we seek the solution on a domain $\Omega$, with suitable initial and boundary conditions. Although some two-dimensional applications will be shown in the numerical experiments at the end of the paper, the theory is developed in the one-dimensional case, in which the computational domain $\Omega$ is an interval, discretized with cells $\Omega_j=[x_{j-1/2}, x_{j+1/2}]$, such that $\cup_j \Omega_j=\Omega$. The amplitude of each cell is $h_j= x_{j+1/2}-x_{j-1/2}$, with the cell centre being $x_j =  (x_{j-1/2}+x_{j+1/2})/2$.

We consider semidiscrete finite volume schemes and denote with  
$$\ca{u}_j(t)=\frac1{h_j} \int_{\Omega_j} u(t,x) \, dx $$
the cell average of the exact solution in the cell $I_j$ at time $t$ and $\ca{U}_j(t)$ its numerical approximation. The semidiscrete numerical scheme can be written as
\begin{equation}\label{e:semischeme}
\frac{\D}{\D t}\ca{U}_j= - \frac{1}{h_j}\left( {F}_{j+1/2}- {F}_{j-1/2}\right) +  G_j(\ca{U},x).
\end{equation}
The numerical fluxes ${F}_{j+1/2}(t)$ should approximate $f(u(t,x_{j+1/2}))$ with suitable accuracy and are computed as a function of the so-called {\em boundary extrapolated data}, i.e.
\begin{equation}\label{e:fluxes}
{F}_{j+1/2}=\mathcal{F}(U_{j+1/2}^-,U_{j+1/2}^+),
\end{equation}
where $\mathcal{F}$ is a consistent and monotone numerical flux, evaluated on two estimates of the solution at the cell interface $U_{j+1/2}^{\pm}$, in turn obtained with a high order non oscillatory reconstruction. The function $\mathcal{F}$ may be an upwind flux, (local) Lax-Friedrichs, an approximate or exact Riemann solver, etc. Finally, $G_j$ is a consistently accurate, well balanced discretization of the cell average of the source term on the cell $\Omega_j$, that in general requires the reconstruction of the point values of the solutions also at inner points in the cell (see e.g. \cite{NatvigEtAl}).

One of the most successful high order reconstructions is the Weighted Essentially Non-Oscillatory (WENO), whose first efficient implementation was described in \cite{JiangShu:96}. It is based on the observation that one may recover the value of high order interpolating polynomials on a centred stencil as convex combination of the values of lower order ones that interpolate the function only in a substencil (see \cite{Gerolymos:12} for a comprehensive development of this idea). 
In particular, let $P_{j,\lambda}(x)$ candidate polynomial in the $j$-th cell. 
Typically, $P_{j,\lambda}(x)$ will interpolate, in the sense of cell averages, the data $\ca{U}$ in a stencil $\lambda\in\Lambda$, where $\Lambda$ is the set of stencil considered in the reconstruction. 

Next, depending on the particular reconstruction being performed, one has one or more sets of {\em optimal weights} $C_\lambda$ such that the linear combination $\sum_\lambda C_\lambda P_{j,\lambda}(x)$ has some desirable property, like coinciding with an higher order reconstruction at some point in the $j$-th cell.
For example, the WENO reconstruction is characterised by two sets of optimal weights $C^{\pm}_\lambda$ such that the combination of the evaluated polynomials $\sum_{\lambda\in\Lambda}C^{\pm}_\lambda P_\lambda(x_j\pm\tfrac{h}{2})$ matches the value of the central reconstruction evaluated at the boundaries. The CWENO reconstruction instead has only one set of optimal weights and computes the linear combination of the candidate polynomials $\sum_{\lambda\in\Lambda}C_\lambda P_\lambda(x)$ which can then be evaluated where needed.

The weights of this convex combination are named {\em linear weights} but the actual reconstruction is computed by a modified set of weights, the so-called {\em nonlinear weights}, that are designed to be close to the linear ones in regions of smoothness and close to zero if a discontinuity occurs in the corresponding substencil. In order to achieve this, one considers suitable smoothness indicators, like those of \cite{JiangShu:96} that defined
\begin{equation}
\label{eq:ShuInd}
I_{j,\lambda} = \sum_{r\geq 1} (h_j)^{2r-1} \int_{\Omega_j} \left[P^{(r)}_{j,\lambda}(x)\right]^2 \D x.
\end{equation}
Note that these indicators are bounded even if the $\lambda$-th stencil contain a discontinuity, while they approach $0$ if the solution is smooth in the stencil.
In order to perform the reconstruction, one forms the {\em nonlinear weights} $\omega_\lambda$, with the help of the regularity indicators as in
\begin{equation}
\label{eq:omega}
\widetilde{\omega}_\lambda = \frac{C_\lambda}{(\epsilon+I_{j,\lambda})^{\tau}},
\qquad
\omega_\lambda = \frac{\widetilde{\omega}_\lambda}{\sum_{\xi\in\Lambda} \widetilde{\omega}_\xi}
\end{equation}
and considers the linear combination $\sum_\lambda \omega_\lambda P_{j,\lambda}(x)$. The final result is a reconstruction that is close to some higher order one (dictated by the choice of the linear weights, often coinciding with the central one) in regions of smoothness and degrades smoothly to a one-sided non-oscillatory one close to discontinuities. Despite the absence of a formal proof of convergence for non- smooth solutions, this basic idea has been very successful and many details related to the very high order versions, the extension to higher dimensions, to non structured meshes, the existence and positivity of the linear weights, together with dozens of applications, have been studied over the years. A comprehensive review with lots of useful references may be found in \cite{Shu:2009:WENOreview}.

In order to obtain a fully discrete scheme, we apply a Strong Stability Preserving Runge-Kutta method (SSPRK) with Butcher's tableau $(A,b)$, obtaining the evolution equation for the cell averages
\begin{equation}
\label{eq:fullydiscrete}
\ca{U}_j^{n+1} = \ca{U}_j^{n} - \frac{\DT}{h_j} \sum_{i=1}^s b_i \left(F^{(i)}_{j+1/2}-F^{(i)}_{j-1/2}\right)
+ \DT \sum_{i=1}^s b_i G^{(i)}_j.
\end{equation}
Here $F^{(i)}_{j+1/2}=\mathcal{F}\big(U^{(i),-}_{j+1/2},U^{(i),+}_{j+1/2}\big)$ and the boundary extrapolated data 
$U^{(i),\pm}_{j+1/2}$ are computed from the stage values of the cell averages
\[
\ca{U}_j^{(i)} =
\ca{U}_j^{n} - \frac{\DT}{h_j} \sum_{k=1}^{i-1} a_{ik} \left(F^{(k)}_{j+1/2}-F^{(k)}_{j-1/2}\right)
+ \DT \sum_{k=1}^{i-1} a_{ik}G^{(k)}_j.
\]
We point out that the spatial reconstruction procedures and the well-balanced quadratures for the source term  must be applied for each stage value of the Runge-Kutta scheme.
In this paper we consider a uniform timestep over the whole grid. A local timestep, keeping a fixed CFL number over the grid, can be enforced using techniques from \cite{Kir:2003,CS:2007,PS:entropy}. The overall accuracy of the scheme is the smallest between the accuracy of the spatial reconstruction and that of scheme used for the time-evolution.

With regards to the accuracy of the WENO reconstructions, \cite{HAP:2005:mappedWENO} pointed out a previously unnoticed phenomenon of loss of accuracy close to local extrema with non-vanishing third derivative and proposed a post processing (called mapping) of the nonlinear weights aimed at closing the gap between linear and nonlinear weights for very smooth functions and recover the full accuracy also in this exceptional situation. Their original technique suitable for fifth order schemes, together with later improvements that extend it to the case of very high order schemes, is now known as WENO-M (see \cite{FengHuangWang:14} and references therein). Another approach to solve the same problem, known as WENO-Z, consists in computing additional indicators, using the Jiang-Shu ones as building blocks, and it is effective from orders higher than 5 (see \cite{DB:2013} and references therein).

Recently, \cite{Arandiga} proposed a completely different technique. Upon observing that in the Jiang-Shu formulation \cite{JiangShu:96}, the occurrence or not of the loss of accuracy is controlled by the relative size of the smoothness indicators and of the constant $\epsilon$, they proposed to set $\epsilon=h^2$ and showed the optimality of this choice in the case of uniform meshes, for WENO reconstructions of arbitrary order of accuracy. Indeed their technique allows to recover the correct asymptotic convergence rates even close to local extrema. Moreover, at the same time, it stabilises the convergence rates close to the asymptotic ones already for very coarse meshes, thus providing a reconstruction scheme that converges at the expected rate already on the meshes used in practice for computations and not only on very fine meshes.
\cite{Kolb2014} extended this idea to the third order CompactWENO (CWENO3) scheme of \cite{LPR:99:1d}, which was originally introduced in the context of central schemes since for the standard WENO3 scheme  one cannot define linear weights that can yield third order accurate reconstructions at the cell centres. Of course the reconstruction is useful also in the case of non-central schemes for balance laws, when the reconstruction at the cell centre is needed in the well-balanced quadrature \cite{NatvigEtAl,PS:shentropy}. The CWENO reconstruction was extended to higher order \cite{Capdeville:08} and higher space dimensions \cite{LPR:02:CWENO2d4thorder,LP:12:CWENO3d}, in the uniform grid setting and to non-uniform meshes of quad-tree type \cite{CRS}, where the CWENO approach is very convenient since it is based on linear weights that do not depend on the relative size of the neighbours. In this last paper, the choice $\epsilon=h$ was found useful in numerical experiments.
WENO-M and WENO-Z are targeted to very high order reconstructions on uniform meshes, but in applications requiring local mesh refinement and unstructured meshes, due to the variability of the disposition and size of the neighbours and the fact that the mesh may change at every timestep, reconstructions with very small stencils and order up to four are to be favoured with respect to those requiring the collection of information over large stencils.
This paper is devoted to the analysis of the non-uniform mesh case and compares the choices $\epsilon=h^2$ and $\epsilon=h$ in the WENO3 and CWENO3 schemes in one space dimension. Several aspects are considered, from the formal accuracy, to stability, to the experimental orders of convergence observed on both coarse and fine meshes.

In Section \S\ref{sec:weno} and \S\ref{sec:cweno} we analyse the effects of different choices of $\epsilon$ in the case of WENO3 and, respectively, CWENO3 schemes on non-uniform meshes. In \S\ref{sec:numerical} we present numerical simulations corroborating our findings, including also tests on balance laws and an h-adaptive scheme for conservation laws. In \S\ref{sec:conclusions} we draw some conclusions and illustrate perspectives for future work.

\section{WENO3}\label{sec:weno}
The third order WENO reconstruction \cite{Shu97} considers the two first-neighbours of the cell in which the reconstruction is to be computed. We thus consider a non-uniform grid with $h_j=h$, $h_{j-1}=\beta h$ and $h_{j+1}=\gamma h$. 
The two candidate polynomials are the linear polynomials matching the cell average of the central cell and of one of the neighbours, namely
\begin{subequations}
\label{eq:WENO:Plinear}
\begin{align}
P_{j,L}(x) &= \ca{u}_j + \sigma_{j,-} (x-x_j),\\
P_{j,R}(x) &= \ca{u}_j + \sigma_{j,+} (x-x_j),
\end{align}
\end{subequations}
where $\sigma_{j,\pm}$ are given by
\[
\sigma_{j,-} = 2\frac{\ca{u}_j-\ca{u}_{j-1}}{(1+\beta)h},
\qquad
\sigma_{j,+} = 2\frac{\ca{u}_{j+1}-\ca{u}_{j}}{(1+\gamma)h}.
\]

The central reconstruction is the second order polynomial whose cell averages in $\Omega_j$ and $\Omega_{j\pm1}$ match $\ca{u}_j$ and $\ca{u}_{j\pm1},$ respectively, and is given by
\begin{equation}\label{eq:P2}
P^{\text{OPT}}_j(x) = a + b(x-x_j)+c(x-x_j)^2,
\end{equation}
where
\begin{gather*}
a=\ca{u}_j-\frac{c}{12}h^2,\\
b=\frac{(\tfrac12+\beta)\sigma_{j,+}+(\tfrac12+\gamma)\sigma_{j,-}}{1+\beta+\gamma},\\
c=\frac32 \frac{\sigma_{j,+}-\sigma_{j,-}}{h(1+\beta+\gamma)}.
\end{gather*}
One may easily verify that
\[
P^{\text{OPT}}_j(x_{j+1/2}) = C^+_{j,L} P_{j,L}(x_{j+1/2}) + C^+_{j,R} P_{j,R}(x_{j+1/2})
\]
for 
\[ 
C^+_{j,L} = \frac{\gamma}{1+\beta+\gamma},
\qquad
C^+_{j,R} = \frac{1+\beta}{1+\beta+\gamma}
\]
and similarly that the optimal weights for the reconstruction at $x_{j-1/2}$ are
\[
C^-_{j,L} = \frac{1+\beta}{1+\beta+\gamma},
\qquad
C^-_{j,R} = \frac{\gamma}{1+\beta+\gamma}.
\]
Note that, in the case of a uniform mesh, $\beta=\gamma=1$ and one recovers the weights computed in \cite{Shu97}
and \cite{Arandiga}.

In order to investigate the accuracy of the reconstruction, we apply it to the cell averages of a regular function $u(x)$, whose Taylor expansions are given by
\begin{align*}
\ca{u}_j &= u_j + \tfrac{h^2}{24}\us + \tfrac{h^4}{1920}\uq + O(h^6),
\\
\ca{u}_{j+1 }&= u_j + \tfrac{h}{2}(\gamma +1) \up + \tfrac{h^2}{24} (4 \gamma^2 +6 \gamma +3)\us + \tfrac{h^3}{48}(\gamma +1)  (2 \gamma^2 +2 \gamma +1) \ut + O(h^4),
\\
\ca{u}_{j-1} &= u_j - \tfrac{h}{2}(\beta +1) \up + \tfrac{h^2}{24} (4 \beta^2 +6 \beta +3)\us - \tfrac{h^3}{48}(\beta +1)  (2 \beta^2 +2 \beta +1) \ut + O(h^4).
\end{align*}

The accuracy of the reconstruction of smooth data depends on  the rate of convergence to zero of the discrepancy between the optimal and the nonlinear weights.
\begin{lemma}
If $ u$ is smooth enough in the stencil $\{\Omega_{j-i}, \Omega_j, \Omega_{j+1}\}$, then 
\begin{equation}
\label{eq:Taylor_ind}
\begin{aligned}
I_{j,L}   & = \frac{4}{(1+\beta)^2}(\ca{u}_j-\ca{u}_{j-1})^2 =  \\
& = {u^\prime_j}^2h^2 - \frac13 (2\beta+1)\up\us h^3+
\frac{1}{36} [ 3 ( 2 \beta^2 + 2 \beta +1) \up \ut + (4 \beta^2 + 4 \beta +1) {u^{\prime\prime}_j}^2] h^4 + \bigO(h^5),
\\
I_{j,R}   & = \frac{4}{(1+\gamma)^2}(\ca{u}_{j+1}-\ca{u}_{j})^2 = \\
&= {u^\prime_j}^2h^2 + \frac13 (2\gamma+1)\up \us h^3+
\frac{1}{36} [ 3 ( 2 \gamma^2 + 2 \gamma +1) \up \ut + (4 \gamma^2 + 4 \gamma +1) {u^{\prime\prime}_j}^2] h^4 + \bigO(h^5).
\end{aligned}
\end{equation}
\end{lemma}

\begin{remark}\label{rem:extrema:WENO}
When the two regularity indicators $I_{j,L}$ and $I_{j,R}$ are very close to each other, procedure \eqref{eq:omega} gives nonlinear weights approximately equal to the linear ones and thus the reconstruction polynomial corresponds to the central second order polynomial. This is highly desirable for the reconstruction of smooth data. On the other hand, when a jump discontinuity is present, the imbalance between the regularity indicators will bias the reconstruction polynomial towards the smoothest stencil (e.g. $\widetilde{\omega}_{j,L}\simeq1$ and $\widetilde{\omega}_{j,R}\simeq0$ if the jump is between $x_j$ and $x_{j+1}$), avoiding oscillations but lowering the accuracy by one order. Unfortunately this may happen also close to local extrema of a smooth function, where both indicators would be close to $0$, but a slight imbalance between them may be amplified by the procedure \eqref{eq:omega} for the computation of nonlinear weights.

Intuitively, if $u'(x_{j+1/2})=0$, $I_{j,R}=(\ca{u}_{j+1}-\ca{u}_{j})^2$ would be much smaller than $I_{j,L}=(\ca{u}_{j}-\ca{u}_{j-1})^2$ and the nonlinear weights will differ quite a lot from the optimal ones, unless $\epsilon$ is chosen such that  it  dominates both $I_{j,R}$ and $I_{j,L}$ in the denominators. Thus a value of $\epsilon$ that does not depend on $h$ will give rise to reconstructions that may change their order of convergence close to local extrema, depending on whether $I_{j,L/R}\gg \epsilon$ or not. The following computations will make this more precise and, for the above mentioned argument, we will concentrate on the case of non-constant $h$-dependent $\epsilon$.
\end{remark}

\begin{theorem}
Let $u$ be a smooth function on a stencil  $\{ \Omega_{i-1}, \ \Omega_i , \ \Omega_{i+1} \},$  with $h_{j-1}= \beta h, \ h_j=h$ and $h_{j+1}= \gamma h.$
Let $\omega^{\pm}_{j,L}, \ \omega^{\pm}_{j,R} $ be the nonlinear weights defined in
\eqref{eq:omega} for the cell $\Omega_j.$ Then,  we have

\begin{equation}  \label{eq:omega_h_h2}
\begin{aligned}
\omega^{+}_{j,L} &= C^{+}_{j,L}\left[
1+\frac{ 2 {\tau}}{3(\hat{\epsilon}+p{\up}^2)} (1+\beta)\up \us h^k+\bigO(h^{k+1})
\right],
\\
\omega^{-}_{j,L} &= C^{-}_{j,L}\left[
1+\frac{ 2 {\tau}}{3(\hat{\epsilon}+p{\up}^2)} \gamma \up \us h^k+\bigO(h^{k+1})
\right],
\\
\omega^{+}_{j,R} &= C^{+}_{j,R}\left[
1-\frac{ 2 {\tau}}{3(\hat{\epsilon}+p{\up}^2)} \gamma \up \us h^k+\bigO(h^{k+1})
\right],
\\
\omega^{-}_{j,R} &= C^{-}_{j,R}\left[
1-\frac{ 2 {\tau}}{3(\hat{\epsilon}+p{\up}^2)} (1+ \beta) \up \us h^k+\bigO(h^{k+1})
\right],
\end{aligned}
\end{equation}
with $ p=0, \   k=2$  if $\epsilon=\hat{\epsilon}h,$ 
$ p=k=1$  if $\epsilon=\hat{\epsilon}h^2.$ 
\end{theorem}
\begin{proof}
Let
\begin{equation}
\label{omegaR}
\widetilde{\omega}^{\pm}_{j,R} = 
\frac{C^{\pm}_{j,R}}{(\epsilon+I_{j,R})^{\tau}},
\end{equation}
with $\tau \ge 2.$
Then, following \cite{Arandiga}, we
  write
\begin{equation}
\label{binom}
\widetilde{\omega}^{\pm}_{j,L} = 
\frac{C^{\pm}_{j,L}}{(\epsilon+I_{j,L})^{\tau}}
=
\frac{C^{\pm}_{j,L}}{(\epsilon+I_{j,R})^{\tau}}
\left[
	1+\frac{I_{j,R}-I_{j,L}}{\epsilon+I_{j,L}}
	\sum_{s=0}^{{\tau}-1} \left( \frac{\epsilon+I_{j,R}}{\epsilon+I_{j,L}}\right)^s
\right].
\end{equation}
By using the Taylor expansions  \eqref{eq:Taylor_ind},
 we obtain 

$$ \begin{aligned}
& \frac{I_{j,R}-I_{j,L}}{\epsilon+I_{j,L}}   = 
 \frac{2(\gamma + \beta +1)}{3(\hat{\epsilon}+p{\up}^2) } \up \us h^k +
  \bigO(h^{k+1})
\end{aligned}$$ 
where   $ k=2, \ p=0 $ in the case $\epsilon=\hat{\epsilon}h$
and $k=p=1 $ if $\epsilon=\hat{\epsilon}h^2,$  thus 
$$
\begin{aligned}
\sum_{s=0}^{{\tau}-1}  \left(\frac{\epsilon+I_{j,R}}{\epsilon+I_{j,L}}\right)^s   & =
\sum_{s=0}^{{\tau}-1}  \left[  1+ s 
\frac{\gamma + \beta +1}{3(\hat{\epsilon}+p{\up}^2) } \up \us h^k + \bigO(h^{k+1}) \right]   \\
 &=  {\tau} + \frac{{\tau}({\tau}-1)}{2}  \frac{\gamma + \beta +1}{3(\hat{\epsilon}+p{\up}^2) } \up \us h^k + \bigO(h^{k+1}). 
\end{aligned}
$$

Now, substituting in the expression \eqref{binom}, we have
\begin{equation}
\label{omegaL}
\widetilde{\omega}^{\pm}_{j,L} = 
\frac{C^{\pm}_{j,L}}{(\epsilon+I_{j,R})^{\tau}} \left[ 1 +  
2 \frac{\gamma + \beta +1}{3(\hat{\epsilon}+p{\up}^2) } {\tau}  \up \us h^k + \bigO(h^{k+1}) \right]
\end{equation}
and
\begin{equation}
\label{omegasum}
 \frac{1}{\widetilde{\omega}^{\pm}_{j,L} +\widetilde{\omega}^{\pm}_{j,R} }
= (\epsilon+I_{j,R})^{\tau} \left[1 - \frac{ 2 {\tau} \psi}{3(\hat{\epsilon}+p{\up}^2) }  \up \us h^k + \bigO(h^{k+1}) \right],
\end{equation}
with $\psi=\gamma $ on the right of the cell, $\psi= 1 + \beta $ on the left.
Then, by using \eqref{omegaR},  \eqref{omegaL} and \eqref{omegasum},  with a simple computation, we obtain the desired results.
\end{proof}

\begin{remark}\label{rem:referee:WENO}
The result above establishes the Taylor expansion of $\omega_\lambda - C_\lambda$ up to the term that is enough for the proof of the third order accuracy of the reconstruction in Theorem \ref{theo:WENO3:conv}. However, it is also interesting to observe the decay rate of $\omega_\lambda - C_\lambda$ close to a local extrema, i.e. to study the term $\bigO(h^{k+1})$ in equation \eqref{eq:omega_h_h2}.
Examining the proof above (see also \cite{Arandiga} for the finite difference case), one can easily see that the accuracy order of $\omega_\lambda - C_\lambda$ is the same as that of $\frac{I_{j,R}-I_{j,L}}{\epsilon+I_{j,L}}$.
From \eqref{eq:Taylor_ind} we have that 
$$I_{j,R}-I_{j,L} = \frac23  (\gamma + \beta + 1) \up \us h^3
+ \frac1{18} (\gamma - \beta) (\gamma + \beta + 1) ( 3 \up \ut + 2 {\us}^2) h^4+ \bigO(h^5) 
$$
and 
$I_{j, K} = \bigO(h^{2s+2})$, for $K=L,R$, where $s = 0$ if $\up (x_j)  \ne 0$ and $s = 1$ if $\up (x_j) =0, \us (x_j) \neq0$. 

So, for $s=0$, we have  
$$ \frac{I_{j,R}-I_{j,L}}{\hat \epsilon h^{\nu} + I_{j,L}} = \frac{ \bigO(h^{3})}{\hat \epsilon h^{\nu} + \bigO(h^{2})}
=\bigO(h^{3- \nu}), \   \  \nu=0,1,2,
$$
with the case $\epsilon= 10^{-6}$ corresponding to $\nu= 0.$  
On the other hand, for $s=1$ one finds
$$ \frac{I_{j,R}-I_{j,L}}{\hat \epsilon h^{\nu} + I_{j,L}} = \frac{ \bigO(h^{4})}{\hat \epsilon h^{\nu} + \bigO(h^{4})}
=\bigO(h^{4- \nu}), \   \  \nu=0,1,2.
$$
in general, but
$$ \frac{I_{j,R}-I_{j,L}}{\hat \epsilon h^{\nu} + I_{j,L}} = \frac{ \bigO(h^{5})}{\hat \epsilon h^{\nu} + \bigO(h^{4})}
=\bigO(h^{5- \nu}), \   \  \nu=0,1,2.
$$
in the case $ \beta = \gamma $, which always occours on uniform grids.
\end{remark}

We denote  with
\begin{equation}
\begin{aligned}
u^+_{j-1/2}  &=  \omega^-_{j,L} P_{j,L}(x_{j-1/2}) + \omega^-_{j,R} P_{j,R}(x_{j-1/2}),   \\
u^-_{j+1/2}  & = \omega^+_{j,L} P_{j,L}(x_{j+1/2}) + \omega^+_{j,R} P_{j,R}(x_{j+1/2}) , 
\end{aligned}
\end{equation}
 the reconstructed values at the left and right boundary of cell $\Omega_j$ and we prove a result on the accuracy of
these boundary extrapolated values.

\begin{theorem}\label{theo:WENO3:conv}
 Let
 $u$ be a smooth function in the stencil $ \{\Omega_{j-i}, \Omega_j, \Omega_{j+1}\}.$ 
Then, 
$$ \begin{aligned}  
 u(x_{j+1/2})  & = u^-_{j+1/2} + \bigO(h^3)    \\
 u(x_{j-1/2})   & = u^+_{j-1/2} + \bigO(h^3). 
\end{aligned}  $$
\end{theorem}



\begin{proof}
We can compute the reconstruction error in the right side of the cell as
\[
u(x_{j+1/2}) - u^-_{j+1/2}=
\underbrace{u(x_{j+1/2}) - P^{\text{OPT}}_j(x_{j+1/2})}_{\bigO(h^3)} +  P^{\text{OPT}}_j(x_{j+1/2}) -u^-_{j+1/2}.
\]
Indeed, the Taylor expansion gives
\begin{equation}
\label{P2Taylor}
   P^{\text{OPT}}_j(x_{j+1/2}) =  u(x_{j+1/2}) + \frac{\gamma}{24} (\beta +1) \ut h^3 + \bigO(h^4) .
\end{equation}
Then, recalling that $C^+_{j,R}+C^+_{j,L}=\omega^+_{j,R}+\omega^+_{j,L}=1$, and computing
\begin{equation}
\label{P1Taylor}
\begin{aligned}
P_{j,R}(x_{j+1/2})  &  =  u(x_{j+1/2}) + \frac{\gamma}6 \us h^2 +  \bigO(h^3), \\
P_{j,L}(x_{j+1/2})  &  =  u(x_{j+1/2}) -  \frac{\beta +1 }6 \us h^2 +  \bigO(h^3),
\end{aligned}
\end{equation}
we estimate
\begin{multline*}
P^{\text{OPT}}_j(x_{j+1/2}) -u^-_{j+1/2} =\\
\left(C^+_{j,R}-\omega^+_{j,R}\right)P_{j,R}(x_{j+1/2})
+\left(C^+_{j,L}-\omega^+_{j,L}\right)P_{j,L}(x_{j+1/2})=\\
\underbrace{\left(C^+_{j,R}-\omega^+_{j,R}\right) }
_{\bigO(h^k)}
\underbrace{\left(P_{j,R}(x_{j+1/2})-u(x_{j+1/2})\right)}
_{\bigO(h^2)}
+\underbrace{\left(C^+_{j,L}-\omega^+_{j,L}\right)}
_{\bigO(h^k)}
\underbrace{\left(P_{j,L}(x_{j+1/2})-u(x_{j+1/2})\right)}
_{\bigO(h^2)}.
\end{multline*}

Using \eqref{eq:omega_h_h2}  in the above formula we have that $k=1$ in the case of $\epsilon=\hat{\epsilon}h^2$ and $k=2$ when $\epsilon=\hat{\epsilon}h$. Thus, with both choices, the reconstruction error is of order $\bigO(h^3)$.
In a similar way, we can compute the reconstruction error in the left side of the cell $\Omega_j.$
\end{proof}

We note that, when $\epsilon=\hat{\epsilon}h^2$,  all terms in the reconstruction error are of order $\bigO(h^3)$, while, for $\epsilon=\hat{\epsilon}h$, the interpolation error of $P^{\text{OPT}}_j$ dominates and the other terms are of higher order. This remark on the balance between the two sources of error is what makes \cite{Arandiga} state that the choice $\epsilon=\hat{\epsilon}h^2$ is optimal and the computations above extend their result to non-uniform meshes.


Next, we want to stress an important difference between finite difference and finite volume schemes. In fact, the accuracy of finite difference schemes depends on the accuracy of the derivative $\tfrac1h(F_{j+1/2}-F_{j-1/2})$ which should be close to $\partial_xf(u(x_{j}))$, while semi-discrete finite volume schemes have their accuracy controlled by the discrepancy between $F_{j+1/2}$ and $f(u(x_{j+1/2}))$, see \cite[\S 17]{LeVeque99}. In the following, we remark that, on uniform meshes, we have the correct accuracy of the derivative in the case of the linear transport equation, while in the numerical tests we will show that this may not maintained on non-uniform meshes, without affecting the convergence order of the fully discrete finite volume scheme.

\begin{remark} \label{rem:uniform:WENO}
We have that in general
\[
u^-_{j+1/2}-u^-_{j-1/2} =
u(x_{j+1/2})-u(x_{j-1/2}) + \bigO(h^3),
\]
but it can  be proven that there is an increased accuracy in the case of a uniform grid. In fact
\begin{align*}
u^-_{j+1/2}-u^-_{j-1/2} = &
\left(P^{\text{OPT}}_j(x_{j+1/2})-P^{\text{OPT}}_{j-1}(x_{j-1/2})\right)\\
&+\left(P^{\text{OPT}}_{j-1}(x_{j-1/2}) -u^-_{j-1/2}\right)
-\left(P^{\text{OPT}}_j(x_{j+1/2})-u^-_{j+1/2}\right)
\end{align*}
and for a uniform grid
\begin{equation} \label{eq:interp2}
P^{\text{OPT}}_j(x_{j+1/2})-P^{\text{OPT}}_{j-1}(x_{j-1/2}) = u(x_{j+1/2})- u(x_{j-1/2}) + \bigO(h^4),
\end{equation}
which raises the question if also the other terms can match this accuracy.

\begin{multline*}
\left(P^{\text{OPT}}_{j-1}(x_{j-1/2}) -u^-_{j-1/2}\right)
-\left(P^{\text{OPT}}_j(x_{j+1/2})-u^-_{j+1/2}\right)=
\\
\left(C^+_{j-1,R}-\omega^+_{j-1,R}\right)P_{j-1,R}(x_{j-1/2})
+\left(C^+_{j-1,L}-\omega^+_{j-1,L}\right)P_{j-1,L}(x_{j-1/2})
\\
-\left(C^+_{j,R}-\omega^+_{j,R}\right)P_{j,R}(x_{j+1/2})
-\left(C^+_{j,L}-\omega^+_{j,L}\right)P_{j,L}(x_{j+1/2})
\end{multline*}
and introducing the Taylor expansion of the approximation error of the linear polynomials the above expression becomes
\begin{align*}
&\left(C^+_{j-1,R}-\omega^+_{j-1,R}\right)\left(u(x_{j-1/2})+\tfrac16\us h^2 +\bigO(h^3)\right)
\\
&+\left(C^+_{j-1,L}-\omega^+_{j-1,L}\right)\left(u(x_{j-1/2})-\tfrac13\us h^2 +\bigO(h^3)\right)
\\
&-\left(C^+_{j,R}-\omega^+_{j,R}\right) \left(u(x_{j+1/2})+\tfrac16\us h^2 +\bigO(h^3)\right)
\\
&-\left(C^+_{j,L}-\omega^+_{j,L}\right) \left(u(x_{j+1/2})-\tfrac13\us h^2 +\bigO(h^3)\right).
\end{align*}
Finally, collecting terms for each power of $h$ and recalling that $C^+_{j,R}+C^+_{j,L}=1=\omega^+_{j,R}+\omega^+_{j,L}$, we find that the the approximation error is of order $\bigO(h^4)$ if
\begin{equation}\label{eq:henrick}
\left(\omega^+_{j,R}-2\omega^+_{j,L}\right) - \left(\omega^+_{j-1,R}-2\omega^+_{j-1,L}\right) = \bigO(h^2).
\end{equation}
For $\epsilon=\hat{\epsilon}h$ the difference between linear and nonlinear weights are of order $\bigO(h^2)$, see \eqref{eq:omega_h_h2} and thus the condition is trivially satisfied; furthermore, direct computation with 
 $\epsilon=\hat{\epsilon}h^2$ shows that condition \eqref{eq:henrick} is satisfied also in this case.

Note that condition \eqref{eq:henrick} is analogous to the one found in \cite[equation (29b)]{HAP:2005:mappedWENO} for the fifth order WENO reconstruction for finite difference schemes; there it was not found useful to design an optimal form of the nonlinear weights, but in this lower order case it is always satisfied.
\end{remark}

\section{CWENO3}\label{sec:cweno}
The third order Compact WENO reconstruction (CWENO3) was introduced in \cite{LPR:99:1d} in the context of central schemes, originally in order to overcome the impossibility of finding linear weights for the reconstruction at cell centre using the WENO3 approach. It is of course useful also in non-staggered schemes, for example if a third order accurate reconstruction is needed at the centre of the cell, like  in the well-balanced quadratures of \cite{NatvigEtAl,PS:shentropy}.  The main idea is to choose linear weights $C_{j,L},C_{j,R}\in(0,1)$ such that $C_{j,L}+C_{j,R}<1$ and to define
\begin{equation}\label{eq:P0}
\begin{aligned}
C_{j,0} &= (1-C_{j,L}-C_{j,R}),\\
P_{j,0}(x) & = \Big (P^{\text{OPT}}_j(x)- C_{j,L} P_{j,L}(x) - C_{j,R} P_{j,R}(x) \Big )/C_{j,0},
\end{aligned}
\end{equation}
so that $P^{\text{OPT}}_j(x)=C_{j,0}P_{j,0}(x)+C_{j,L} P_{j,L}(x)+C_{j,R} P_{j,R}(x)$ at every point in the cell. Unlike in the WENO3 setting, here the values of $C_{j,L}$ and $C_{j,R}$ do not have to satisfy any accuracy-related requirement and thus can be in principle chosen arbitrarily; in particular we can employ the same linear weights, independently on the reconstruction point and on the local mesh geometry in the neighbourhood of the cell.

With the choices $C_{j,L}=C_{j,R}=\tfrac14$,  $C_{j,0} = \tfrac12$ of \cite{LPR:99:1d} one has that
\[
P_{j,0}(x)  = \ca{u}_j-\frac{c}{6}h^2 + (2b-\frac{\sigma_{j,+}+\sigma_{j,-}}{2})(x-x_j)+2c(x-x_j)^2.
\]
\begin{lemma}
For $u$ sufficiently smooth in the stencil $ \{\Omega_{j-i}, \Omega_j, \Omega_{j+1}\}$, we have
$$ \begin{aligned}
I_{j,0}   & = h^2 (2b-\frac{\sigma_{j,+}+\sigma_{j,-}}{2})^2 + \tfrac {52}{3} c^2 h^4 = 
 {\up}^2 h^2 + \tfrac 13 (\beta -\gamma) \up \us h^3 + \\
 & + \left[ \left(\frac{13}3 + \frac1{36} (\beta - \gamma)^2\right) {\us}^2
- \frac1{12}( \beta^2 + \gamma^2 - 4 \beta \gamma - \beta  -\gamma -1) \up 
\ut \right] h^4 + \bigO(h^5).  
\end{aligned} $$
\end{lemma}

\begin{theorem}
Let $u$ be a smooth function on a stencil $\{\Omega_{j-i}, \Omega_j, \Omega_{j+1}\}$ with $h_{j-1}= \beta h, \ h_j=h$ and $h_{j+1}= \gamma h.$
We denote by $\omega_{j,L}, \ \omega_{j,R}, \  \omega_{j,0}  $ the nonlinear weight defined in
\eqref{eq:omega} in the cell $\Omega_j.$ Then, 
\begin{equation}\label{eq:CWh_h2}
\begin{aligned}
\omega_{j,L} &= C_{j,L}\left[
1+\frac{2 \beta +1 }{3 (\hat{\epsilon}+p{\up}^2)}  \,  {\tau} \, \up \us h^k+\bigO(h^{k+1})
\right],
\\
\omega_{j,R} &= C_{j,R}\left[
1-\frac{ 2 \gamma  +1 }{3 (\hat{\epsilon}+p{\up}^2)}  \,  {\tau} \,  \up \us h^k+\bigO(h^{k+1})
\right],
\\
\omega_{j,0} &= C_{j,0}\left[
1+\frac{ \gamma - \beta }{3 (\hat{\epsilon}+p{\up}^2)}  \,  {\tau} \,   \up \us h^k+\bigO(h^{k+1})
\right],
\end{aligned}
\end{equation}
with $ p=0, \   k=2$ for  $\epsilon=\hat{\epsilon}h,$  and $p=k=1$  for $\epsilon=\hat{\epsilon}h^2.$ 

\end{theorem}
\begin{proof}
Following the same procedure employed to get \eqref{eq:omega_h_h2}  in the WENO3 case, we obtain the Taylor expansions for the nonlinear weights. In particular, for $ \omega_{j,0}$, we find,
\begin{equation}
\label{binomCW}
\widetilde{\omega}_{j,0} = 
\frac{C_{j,0}}{(\epsilon+I_{j,0})^{\tau}}
=
\frac{C_{j,0}}{(\epsilon+I_{j,R})^{\tau}}
\left[
	1+\frac{I_{j,R}-I_{j,0}}{\epsilon+I_{j,0}}
	\sum_{s=0}^{{\tau}-1} \left( \frac{\epsilon+I_{j,R}}{\epsilon+I_{j,0}}\right)^s
\right].
\end{equation}
By using the Taylor expansions  \eqref{eq:Taylor_ind},  we obtain $$\frac{I_{j,R}-I_{j,0}}{\epsilon+I_{j,0}}= 
\frac{3 \gamma -  \beta +1}{3 (\hat{\epsilon}+p{\up}^2) } \up \us h^k + \bigO(h^{k+1})$$ 
where  $k=2, \ p=0 $ if $\epsilon=\hat{\epsilon}h$
and $k=p=1 $ if $\epsilon=\hat{\epsilon}h^2.$  Thus,
$$
\begin{aligned}
\sum_{s=0}^{{\tau}-1}  \left(\frac{\epsilon+I_{j,R}}{\epsilon+I_{j,0}}\right)^s   & =
\sum_{s=0}^{{\tau}-1}  \left[  1+ s 
\frac{3 \gamma -  \beta +1}{3(\hat{\epsilon}+p{\up}^2) } \up \us h^k + \bigO(h^{k+1}) \right]   \\
 &=  \tau + \frac{{\tau}({\tau}-1)}{2}  \frac{3 \gamma -   \beta +1}{3(\hat{\epsilon}+p{\up}^2) } \up \us h^k + \bigO(h^{k+1}). 
\end{aligned}
$$

Now, substituting in \eqref{binomCW}
$$
\widetilde{\omega}_{j,0} = 
\frac{ C_{j,0}}{(\epsilon+I_{j,R})^{\tau}} \left[ 
1 +   \frac{ 3 \gamma  - \beta +1}{3(\hat{\epsilon}+p{\up}^2) } {\tau} \up \us h^k + \bigO(h^{k+1}) \right]
$$
and
$$ \frac{1}{\widetilde{\omega}_{j,L} +\widetilde{\omega}_{j,R} +\widetilde{\omega}_{j,0}}
= (\epsilon+I_{j,R})^{\tau} \left[1 - \frac{2 \gamma +1 }{3(\hat{\epsilon}+p{\up}^2) } {\tau}  \up \us h^k 
+ \bigO(h^{k+1}) \right].$$
Then, recalling \eqref{omegaR}, \eqref{omegaL}, with the different values for $ C_{j,L} $
and $  C_{j,R}$,  we obtain the results  by substituting in
$$ \omega_{j,0}= \frac{\widetilde{\omega}_{j,0}}{\widetilde{\omega}_{j,L} +\widetilde{\omega}_{j,R} +\widetilde{\omega}_{j,0}},  \   \ \ 
\omega_{j,L}= \frac{\widetilde{\omega}_{j,L}}{\widetilde{\omega}_{j,L} +\widetilde{\omega}_{j,R} +\widetilde{\omega}_{j,0}},  \  \  \
\omega_{j,R}= \frac{\widetilde{\omega}_{j,R}}{\widetilde{\omega}_{j,L} +\widetilde{\omega}_{j,R} +\widetilde{\omega}_{j,0}}.
$$
\end{proof}

\begin{remark}\label{rem:referee:CWENO}
Using the same notation of Remark \ref{rem:referee:WENO}, in this case we have that
$$ \begin{aligned}
& I_{j,R}-I_{j,0}  = \frac13 ( 3 \gamma - \beta +1) \up \us h^3 + \\ 
& +  \frac1{12} \left[ ( \gamma - \beta) (3 \gamma - \beta +1)  \up \ut  + 
\frac{ 3 \gamma^2 - \beta^2 + 4 \gamma + 2 \beta \gamma -155}{3} {\us}^2 \right]  h^4 + \bigO(h^5)
\end{aligned} $$
and thus $I_{j,R}-I_{j,0}  = \bigO(h^{3+s})$ independently of the grid employed.
With the same argument of Remark \ref{rem:referee:WENO}, using 
$I_{j, 0} = \bigO(h^{2s+2})$,  one shows that,  when $\up(x_j)=0$,
$C_\lambda-\omega_\lambda=\bigO(h^{4-\nu})$.
\end{remark}

In the CWENO3 reconstruction, only one set of nonlinear weights is computed and then
$u^-_{j+1/2}$ and $u^+_{j-1/2}$ are simply obtained by evaluation of 
$ \widetilde P_j(x)=\omega_{j,L} P_{j,L}(x) + \omega_{j,R} P_{j,R}(x) + \omega_{j,0} P_{j,0}(x)$
at the boundaries of the cell.

\begin{theorem} Let
 $u$ be a smooth function in the stencil $ \{\Omega_{j-i}, \Omega_j, \Omega_{j+1}\}.$ 
Then, 
$$  \begin{aligned}  u(x_{j+1/2}) &  = u^-_{j+1/2} + \bigO(h^3),  \\
u(x_{j-1/2})   & = u^+_{j-1/2} + \bigO(h^3).  \end{aligned} $$
\end{theorem}
\begin{proof}
Similarly to the WENO3 case, from \eqref{P2Taylor},
we can compute the reconstruction error as
\[
u(x_{j+1/2}) - u^-_{j+1/2}=
\underbrace{u(x_{j+1/2}) - P^{\text{OPT}}_j(x_{j+1/2})}_{\bigO(h^3)} +  P^{\text{OPT}}_j(x_{j+1/2}) -u^-_{j+1/2}.
\]
Then, recalling that $C_{j,R}+C_{j,L}+C_{j,0}=\omega_{j,R}+\omega_{j,L}+\omega_{j,0}=1$,  from 
\eqref{P1Taylor},  and using the Taylor expansion of the quadratic  polynomials
\begin{equation}
 P_{j,0}(x_{j+1/2})  = u(x_{j+1/2})+\frac{\beta - \gamma +1}{12}  \us h^2 +\bigO(h^3),
\end{equation}
we estimate
\begin{multline*}
P^{\text{OPT}}_j(x_{j+1/2}) -u^-_{j+1/2} =\\
\left(C_{j,R}-\omega_{j,R}\right)P_{j,R}(x_{j+1/2})
+\left(C_{j,L}-\omega_{j,L}\right)P_{j,L}(x_{j+1/2})+ \left(C_{j,0}-\omega_{j,0}\right)P_{j,0}(x_{j+1/2})=\\
\underbrace{\left(C_{j,R}-\omega_{j,R}\right) }
_{\bigO(h^k)}
\underbrace{\left(P_{j,R}(x_{j+1/2})-u(x_{j+1/2})\right)}
_{\bigO(h^2)}
+\underbrace{\left(C_{j,L}-\omega_{j,L}\right)}
_{\bigO(h^k)}
\underbrace{\left(P_{j,L}(x_{j+1/2})-u(x_{j+1/2})\right)}
_{\bigO(h^2)} + \\
+\underbrace{\left(C_{j,0}-\omega_{j,0}\right)}
_{\bigO(h^k)}
\underbrace{\left(P_{j,0}(x_{j+1/2})-u(x_{j+1/2})\right)}
_{\bigO(h^2)}.
\end{multline*}
Thanks to \eqref{eq:CWh_h2}, in the formula above  $k=1$ in the case of $\epsilon=\hat{\epsilon}h^2$, while $k=2$ when
$\epsilon=\hat{\epsilon}h$. 
\end{proof}

We note that, with both choices, the reconstruction error is of order $\bigO(h^3)$, but,  when $\epsilon=\hat{\epsilon}h^2$,  all terms in the reconstruction error are of order $\bigO(h^3)$, while for $\epsilon=\hat{\epsilon}h$, the interpolation error of $P^{\text{OPT}}_j$ dominates and the other terms are of lower order. These computations extend the results of  \cite{Kolb2014}
to the non-uniform case.

We conclude this section with the analogous of Remark   \ref{rem:uniform:WENO} in the case of the CWENO3 reconstruction. As for the previous case, this result should be contrasted with the numerical evidence of \S \ref{ssec:numerder}.

\begin{remark}\label{rem:uniform:CWENO}
Proceeding as in Remark   \ref{rem:uniform:WENO},  we obtain that, using a uniform grid,
\[
u^-_{j+1/2}-u^-_{j-1/2} =
u(x_{j+1/2})-u(x_{j-1/2}) + \bigO(h^4).
\]
 In the CWENO3 approximation we have

\begin{multline*}
\left(P^{\text{OPT}}_{j-1}(x_{j-1/2}) -u^-_{j-1/2}\right)
-\left(P^{\text{OPT}}_j(x_{j+1/2})-u^-_{j+1/2}\right)=
\\
\left(C_{j-1,R}-\omega_{j-1,R}\right)P_{j-1,R}(x_{j-1/2})
+\left(C_{j-1,L}-\omega_{j-1,L}\right)P_{j-1,L}(x_{j-1/2})
\\
+\left(C_{j-1,0}-\omega_{j-1,0}\right)P_{j-1,0}(x_{j-1/2})
-\left(C_{j,0}-\omega_{j,0}\right)P_{j,0}(x_{j+1/2})
\\
-\left(C_{j,R}-\omega_{j,R}\right)P_{j,R}(x_{j+1/2})
-\left(C_{j,L}-\omega_{j,L}\right)P_{j,L}(x_{j+1/2}).
\end{multline*}
Using the Taylor expansion of the quadratic  polynomials
\begin{align*}
P_{j-1,0}(x_{j-1/2})  & = u(x_{j-1/2})+\tfrac1{12}\us h^2 +\bigO(h^3),
\\
 P_{j,0}(x_{j+1/2}) & = u(x_{j+1/2})+\tfrac1{12}\us h^2 +\bigO(h^3),
\end{align*}
collecting terms for each power of $h$, recalling  \eqref{eq:interp2} and that $C_{j,R}+C_{j,L}+C_{j,0}=\omega_{j,R}+\omega_{j,L}+\omega_{j,0}=1$, we find that the the approximation error is of order $\bigO(h^4)$ if
\begin{equation}\label{eq:henrickCW}
\left(\omega_{j,0}-\omega_{j-1,0}\right)+2\big [ \left(\omega_{j,R}-2\omega_{j,L}\right) - \left(\omega_{j-1,R}-2\omega_{j-1,L}\right)\big ] = \bigO(h^2).
\end{equation}
For $\epsilon=\hat{\epsilon}h$ the difference between linear and nonlinear weights are of order $\bigO(h^2)$, see \eqref{eq:CWh_h2} with $ k=2,$ so the condition is trivially satisfied; furthermore, in the case $\epsilon=\hat{\epsilon}h^2$, the condition \eqref{eq:henrickCW} is satisfied by using  \eqref{eq:CWh_h2} with $k=1$. Note that equation  \eqref{eq:henrickCW} is analogous to   \eqref{eq:henrick}    in the case of CWENO3 reconstruction.
\end{remark}

\section{Numerical tests}\label{sec:numerical}
First we describe the meshes used in the numerical tests, on the reference interval $[0,1]$. {\em Uniform} grids with $N$ cells, have cells of size $h=1/N$ and cell centres $x_{j+1/2}=(j+\tfrac12)h$. A {\em quasi regular} grid is obtained transforming the cell centres of a uniform grid with the map
\[\varphi(x)=x+0.1*\sin(10\pi x)/5.\]
 The grid spacing in quasi-regular grids is asymptotically described by $\varphi^\prime(x)$. In this way, we obtain a grid with cell sizes $h_j$ such that 
\[ 
 (1-\tfrac{\pi}{5})\tfrac1N \leq h_j \leq  (1+\tfrac{\pi}{5})\tfrac1N.
\]
Note that the ratio of consecutive cells approaches $1$:
\begin{equation} \label{eq:gamma}
\frac{h_{k+1}}{h_k}  \simeq \frac{\varphi^\prime(\tfrac{k+1}{N})}{\varphi^\prime(\tfrac{k}{N})}
  \simeq 1+\frac{1}{N}\frac{\varphi^{\prime\prime}(\tfrac{k}{N})}{\varphi^\prime(\tfrac{k}{N})}.
\end{equation}
In the notation of Sections \ref{sec:weno} and \ref{sec:cweno}, this means that, for increasing $N$, the parameters $\beta$ and $\gamma$ approach the value $1$ that characterise uniform meshes. We will observe that the numerical schemes on quasi-regular grids behave very much like on uniform ones.

Next, we consider {\em random} grids that are obtained moving randomly the interfaces of a uniform grid. We start from a uniform grid with cells of size $h=1/N$ and then we consider grids with interfaces at
\[ \tilde{x}_{j+1/2} = {x}_{j+1/2} + \xi_j\tfrac{h}{4}, \]
where $\xi_j$ are random numbers uniformly distributed in $[-0.5,0.5]$. We have
\[ 
\tfrac34 \tfrac1N \leq h_j \leq  \tfrac54 \tfrac1N.
\]
We use this grid for the purpose of illustration even if, of course, one would not use such  irregular grid in an application.

Finally, observe that in binary-tree mesh refinement in $d$ spatial dimensions, one starts with a uniform mesh and then, guided by some error indicator, each cell may be split (recursively) in $2^d$ equal parts, see  \cite{PS:entropy,CRS} . The ratio of the sizes of adjacent cells is thus a positive  power of $2$ and does not approach $1$. In order to test the schemes on meshes that may be employed by such AMR techniques, we consider the {\em $\alpha\beta\gamma\delta$} meshes, that are composed by adjoining building blocks of $5$ cells of size $\alpha h$, $\beta h$, $h$, $\gamma h$, $\delta h$ respectively.

\subsection{Comparing $\epsilon$'s}
The first set of tests is on the reconstruction error of WENO3 and CWENO3 on non-uniform meshes. An $\alpha\beta\gamma\delta$-grid is set up with cells of size $h,2h,h,h/2,h/2$ with $x=0$ at the centre of the middle cell. In all our tests, we consider nonlinear weights \eqref{eq:omega} with $\tau=2.$ For decreasing values of $h$, cell averages of a function $u(x)$ were set using the fifth order gaussian quadrature rule and the reconstruction at the right boundary of the middle cell was computed and compared with the exact values $u(h/2)$.

\begin{table}
\begin{center}
Reconstruction  error for {$u(x)=e^x$}\\
\begin{tabular}{|r|ll|ll|ll|ll|}
\hline
$h$ & \multicolumn{2}{c|}{$\epsilon=10^{-30}$} &
 \multicolumn{2}{c|}{$\epsilon=10^{-6}$}& 
 \multicolumn{2}{c|}{$\epsilon=h$}&
 \multicolumn{2}{c|}{$\epsilon=h^2$}
 \\
& error & rate & error & rate& error & rate& error & rate\\
\hline
5.00e-02 &   1.16e-05 &       &  1.16e-05 &       &  2.31e-06    &        &  4.70e-06 &     \\
   2.50e-02 &  1.43e-06 &  3.02 &  1.42e-06 &  3.03 &  3.08e-07 &  2.91  & 5.65e-07   & 3.06 \\
   1.25e-02 &   1.78e-07  &  3.01 &  1.72e-07 &  3.04 &  3.96e-08 &  2.96 & 6.92e-08   & 3.03 \\
   6.25e-03 &  2.21e-08 &  3.00 &  1.96e-08 &  3.13 &  5.02e-09 &  2.98 & 8.56e-09   & 3.01 \\
   3.12e-03 &  2.76e-09 &  3.00 &  1.78e-09 &  3.47 &  6.32e-10 &  2.99  & 1.07e-09   & 3.01 \\
   1.56e-03 &  3.45e-10 &  3.00 &  8.15e-11 &  4.45 &  7.92e-11 &  3.00 &  1.33e-10  & 3.00  \\
   7.81e-04 &  4.31e-11 &  3.00 &  2.92e-12 &  4.80 &  9.92e-12 &  3.00 & 1.66e-11   & 3.00 \\
   3.90e-04 &  5.38e-12 &  3.00 &  9.99e-13 &  1.55 &  1.24e-12 &  3.00 & 2.07e-12   &  3.00 \\
   1.95e-04 &  6.73e-13 &  3.00 &  1.48e-13 &  2.78 &  1.55e-13 &  3.00 & 2.59e-13   & 3.00 \\
   9.76e-05 &  8.39e-14 &  3.00 &  1.91e-14 &  2.95 &  1.95e-14 &  3.00 & 3.22e-14   & 3.00 \\
\hline
\end{tabular}
\end{center}
\begin{center}
Reconstruction  error for $ u(x)=\cos\left(2\pi x\right)+x^3$\\
\begin{tabular}{|r|ll|ll|ll|ll|}
\hline
$h$ & \multicolumn{2}{c|}{$\epsilon=10^{-30}$} &
 \multicolumn{2}{c|}{$\epsilon=10^{-6}$}& 
 \multicolumn{2}{c|}{$\epsilon=h$}&
 \multicolumn{2}{c|}{$\epsilon=h^2$}
 \\
& error & rate & error & rate& error & rate& error & rate\\
\hline
5.00e-02 &  7.91e-03 &       &  7.91e-03 &       &  7.61e-04    &           & 6.79e-03 &     \\
   2.50e-02 &  2.00e-03 &  1.99 &  1.99e-03 &  1.99 &  3.12e-05 &  4.61  & 1.06e-03   & 2.68 \\
   1.25e-02 &  5.01e-04 &  2.00 &  4.75e-04 &  2.07 &  1.41e-06 &  4.47 & 9.72e-05   & 3.45 \\
   6.25e-03 &  1.25e-04 &  2.00 &  4.91e-05 &  3.28 &  8.19e-08 &  4.10 & 6.77e-06   & 3.84 \\
   3.12e-03 &  3.13e-05 &  2.00 &  1.04e-06 &  5.55 &  6.35e-09 &  3.69  & 4.36e-07   & 3.96 \\
   1.56e-03 &  7.84e-06 &  2.00 &  1.71e-08 &  5.93 &  6.14e-10 &  3.37 &  2.77e-08  & 3.98 \\
   7.81e-04 &  1.96e-06 &  2.00 &  3.26e-10 &  5.72 &  6.75e-11 &  3.19 & 1.76e-09   & 3.97 \\
   3.90e-04 &  4.90e-07 &  2.00 &  1.20e-11 &  4.77 &  7.92e-12 &  3.09 & 1.14e-10  &  3.95 \\
   1.95e-04 &  1.22e-07 &  2.00 &  1.02e-12 &  3.55 &  9.60e-13 &  3.04 & 7.59e-12   & 3.91 \\
   9.76e-05 &  3.06e-08 &  2.00 &  1.19e-13 &  3.10 &  1.18e-13 &  3.01 & 5.33e-13   & 3.83 \\
\hline
\end{tabular}
\end{center}
\caption{WENO3. Reconstruction errors at $x=0+h/2$ for a grid of five cells of size $h,2h,h,h/2,h/2$ with $x=0$ in the centre of the middle cell. In the first test $u^\prime\neq0$ in the reconstruction stencil, while $u^\prime(0)=0$ in the second case.}
\label{tab:WENO3:epsilon}
\end{table}

In Table \ref{tab:WENO3:epsilon} we show the reconstruction  errors and convergence rates for WENO3 on non-uniform meshes, with different choices of $\epsilon$.
It is clear that the non-uniformity of the mesh has no influence on the convergence rates. In fact, as in the uniform grid tests of \cite{Arandiga}, an irregular convergence rate appears for constant $\epsilon$ when $u'$ vanishes in the central cell, but a regular convergence history can be recovered by employing an $h$-dependent $\epsilon$. We remark that $\epsilon=h$ gives slightly lower errors than $\epsilon=h^2$. We also point out that  repeating the test of Table \ref{tab:WENO3:epsilon}  for a function such that $u'(h/2)=0,$ gives analogous results, indicating that convergence may be degraded whenever $u'$ vanishes in the reconstruction stencil.

\begin{table}
\begin{center}
Reconstruction  error for $  u(x)={e^x}$\\
\begin{tabular}{|r|ll|ll|ll|ll|}
\hline
$h$ & \multicolumn{2}{c|}{$\epsilon=10^{-30}$} &
 \multicolumn{2}{c|}{$\epsilon=10^{-6}$}& 
 \multicolumn{2}{c|}{$\epsilon=h$}&
 \multicolumn{2}{c|}{$\epsilon=h^2$}
 \\
& error & rate & error & rate& error & rate& error & rate\\
\hline
5.00e-02 &   4.60e-06 &       &  4.59e-06 &       &  2.50e-06    &        &  1.05e-06 &     \\
   2.50e-02 &  5.58e-07 &  3.04 &  5.53e-07 &  3.05 &  3.19e-07 &  2.97  & 1.19e-07   & 3.13 \\
   1.25e-02 &   6.88e-08  &  3.02 &  6.61e-08 &  3.06 &  4.03e-08 &  2.99 & 1.42e-08   & 3.07 \\
   6.25e-03 &  8.54e-09 &  3.01 &  7.28e-09 &  3.18 &  5.06e-09 &  2.99 & 1.74e-09   & 3.03 \\
   3.12e-03 &  1.06e-09 &  3.01 &  5.70e-10 &  3.67 &  6.34e-10 &  3.00  & 2.15e-10   & 3.02 \\
   1.56e-03 &  1.33e-10 &  3.00 &  9.70e-13 &  9.20 &  7.94e-11 &  3.00 &  2.67e-11  & 3.01  \\
   7.81e-04 &  1.66e-11 &  3.00 &  6.43e-12 &  -2.73 &  9.93e-12 &  3.00 & 3.32e-12   & 3.00 \\
   3.90e-04 &  2.07e-12 &  3.00 &  1.12e-12 &  2.52 &  1.24e-12 &  3.00 & 4.15e-13   &  3.00 \\
   1.95e-04 &  2.59e-13 &  3.00 &  1.52e-13 &  2.88 &  1.55e-13 &  3.00 & 5.15e-14   & 3.00 \\
   9.76e-05 &  3.22e-14 &  3.00 &  1.95e-14 &  2.96 &  1.95e-14 &  3.00 & 6.44e-15   & 3.00 \\
\hline
\end{tabular}
\end{center}
\begin{center}
Reconstruction  error for $ u(x)=\cos\left(2\pi x\right)+x^3$\\
\begin{tabular}{|r|ll|ll|ll|ll|}
\hline
$h$ & \multicolumn{2}{c|}{$\epsilon=10^{-30}$} &
 \multicolumn{2}{c|}{$\epsilon=10^{-6}$}& 
 \multicolumn{2}{c|}{$\epsilon=h$}&
 \multicolumn{2}{c|}{$\epsilon=h^2$}
 \\
& error & rate & error & rate& error & rate& error & rate\\
\hline
5.00e-02 &  7.85e-03 &       &  7.85e-03 &       &  4.81e-04    &           & 6.38e-03 &     \\
   2.50e-02 &  1.98e-03 &  1.99 &  1.98e-03 &  1.99 &  2.05e-05 &  4.56  & 8.49e-04   & 2.91 \\
   1.25e-02 &  4.97e-04 &  2.00 &  4.64e-04 &  2.09 &  1.07e-06 &  4.27 & 6.06e-05   & 3.81 \\
   6.25e-03 &  1.24e-04 &  2.00 &  3.58e-05 &  3.70 &  7.11e-08 &  3.91 & 3.65e-06   & 4.05 \\
   3.12e-03 &  3.11e-05 &  2.00 &  5.48e-07 &  6.03 &  6.01e-09 &  3.56  & 2.25e-07   & 4.02 \\
   1.56e-03 &  7.78e-06 &  2.00 &  8.89e-09 &  5.94 &  6.04e-10 &  3.31 &  1.42e-08  & 3.98 \\
   7.81e-04 &  1.94e-06 &  2.00 &  1.96e-10 &  5.50 &  6.72e-11 &  3.17 & 9.16e-10   & 3.95 \\
   3.90e-04 &  4.86e-07 &  2.00 &  9.93e-12 &  4.31 &  7.92e-12 &  3.09 & 6.10e-11   &  3.91 \\
   1.95e-04 &  1.22e-07 &  2.00 &  9.91e-13 &  3.32 &  9.60e-13 &  3.04 & 4.28e-12   & 3.83 \\
   9.76e-05 &  3.04e-08 &  2.00 &  1.19e-13 &  3.06 &  1.18e-13 &  3.01 & 3.25e-13   & 3.72 \\
\hline
\end{tabular}
\end{center}

\caption{CWENO3. Reconstruction errors at $x=0+h/2$ for a grid of five cells of size $h,2h,h,h/2,h/2$  with $x=0$ in the centre of the middle cell. In the first test $u^\prime\neq0$ in the reconstruction stencil, while $u^\prime(0)=0$ in the second case.}
\label{tab:CWENO3:epsilon}
\end{table}

In Table \ref{tab:CWENO3:epsilon} we show the same tests for CWENO3; we obtained analogous results and the remarks about the comparison with the uniform grid case could be repeated.

 In  Figure \ref{fig:Comega},  we show the distance $ \vert \omega_{\lambda} - C_{\lambda} \vert $
observed when reconstructing $ u(x)= x^3 + \cos x$ on a non-uniform grid of type $\alpha,\beta,\gamma,\delta.$ 
 As expected, in both the WENO3 and CWENO3 cases, the choice $\epsilon = 10^{-30}$ does not give weights
converging to their optimal values, the choice $\epsilon = 10^{-6}$ behaves similarly on coarse meshes and changes to a convergent regime at about $h=2 \times 10^{-3}.$ On the other hand, the choices $ \epsilon= h$ and $ \epsilon=h^2$
give rise to a more regular convergence histories,
with the former giving lower discrepancies between nonlinear  and optimal weights.

\begin{table}
\begin{center}
\begin{tabular}{|c|cc|cc|cc|}                            
\hline
& \multicolumn{6}{c|}{$u(x)= x^3 + \cos(2 \pi x)$}  \\
\hline
 & \multicolumn{2}{c|} {$\epsilon=10^{-6}$} & \multicolumn{2}{c|} {$\epsilon=h$}  & \multicolumn{2}{c|}{$\epsilon=h^2$}
 \\
$h$ & $C^+_R-\omega^+_R$ & rate & $C^+_R-\omega^+_R$ & rate& $C^+_R-\omega^+_R$ & rate\\
\hline
5.00e-02   &    1.38e-01 &           &  2.27e-02 &        & 1.67e-01 &        \\
2.50e-02  & 1.38e-01 &   0.00  &3.37e-03 &   2.75    & 7.45e-02 &  0.65  \\
1.25e-02   &1.37e-01 &  0.02  &  4.33e-04 & 2.96  & 2.89e-02 & 1.37 \\
  6.25e-03   &1.06e-01 &  0.36 &  5.44e-05 & 2.99    & 8.29e-03 & 1.80  \\
3.12e-03   &1.89e-02 &  2.49  & 6.80e-06 &  3.00     &2.15e-03 &  1.95   \\
 1.56e-03  & 1.32e-03 &  3.84 & 8.50e-07 &  3.00 &  5.42e-04 &  1.99 \\
7.81e-04   & 8.29e-05 &  3.99  & 1.06e-07 &  3.00  &  1.36e-04 & 2.00   \\
3.90e-04  & 5.18e-06&  4.00  & 1.33e-08 &  3.00  & 3.40e-05 & 2.00  \\
  1.95e-04 &3.24e-07 &  4.00  & 1.66e-09 &  3.00   & 8.49e-06 & 2.00   \\
   9.76e-05  &2.02e-08 &   4.00 & 2.07e-10 &  3.00  & 2.12e-06 & 2.00   \\
\hline
\end{tabular}

\begin{tabular}{|c|cc|cc|cc|}                            
\hline
& \multicolumn{6}{c|}{$u(x)= e^x$}  \\
\hline
 & \multicolumn{2}{c|}{$\epsilon=10^{-6}$} & \multicolumn{2}{c|} {$\epsilon=h$}  & \multicolumn{2}{c|}{$\epsilon=h^2$}
 \\
$h$ & $C^+_R-\omega^+_R$ & rate & $C^+_R-\omega^+_R$ & rate& $C^+_R-\omega^+_R$ & rate\\
\hline
5.00e-02   &    3.08e-02 &           &  1.33e-03 &        & 1.47e-02 &        \\
2.50e-02  & 1.48e-02 &   1.06  &3.44e-04 &   1.95     & 7.22e-03 &  1.02  \\
1.25e-02   &7.23e-03 &  1.03  &  8.75e-05 & 1.97  & 1.79e-03 & 1.01  \\
  6.25e-03   &3.51e-03 &  1.04 &  2.21e-05 & 1.99    & 8.94e-04 & 1.00  \\
3.12e-03   &1.63e-03 &  1.11  & 5.55e-06 &  1.99     &4.47e-04 &  1.00   \\
 1.56e-03  & 6.34e-04 &  1.36 & 1.39e-06 &  2.00 &  2.23e-04 &  1.00 \\
7.81e-04   & 1.69e-04 &  1.91  & 3.48e-07 &  2.00  &  1.12e-04 & 1.00   \\
3.90e-04  & 2.95e-05&  2.52  & 8.71e-08 &  2.00  & 5.58e-05 & 1.00 \\
  1.95e-04 &4.10e-06 &  2.85  & 2.18e-08 &  2.00   & 2.79e-05 & 1.00   \\
   9.76e-05  &5.27e-07 &   2.99 & 5.45e-09 &  2.00  & 1.39e-05 & 1.00   \\
\hline
\end{tabular}
\end{center}
\caption{WENO3. Distance and order of convergence of the optimal weights $C^+_R$ and the nonlinear weights $\omega^+_R$ , as a function of  the mesh width $h$, in the cell centered in $x=0$.}
\label{tab:WENO3:Comega}
\end{table}

\begin{table}

\begin{center}

\begin{tabular}{|c|cc|cc|cc|}                            
\hline
& \multicolumn{6}{c|}{$u(x)= x^3 + \cos(2 \pi x)$}  \\
\hline
$h$ & \multicolumn{2}{c|} {$\epsilon=10^{-6}$} & \multicolumn{2}{c|} {$\epsilon=h$}  & \multicolumn{2}{c|}{$\epsilon=h^2$}
 \\
& $C_0-\omega_0$ & rate & $C_0-\omega_0$ & rate& $C_0-\omega_0$ & rate\\
\hline
5.00e-02   &    4.99e-01 &           &  2.43e-01 &        & 4.89e-01 &        \\
2.50e-02  & 4.99e-01 &   0.00  &4.55e-02 &   2.42    & 4.42e-01 &  0.15  \\
1.25e-02   &4.98e-01 &  0.00  &  6.02e-03 & 2.92  & 2.82e-01 & 0.65 \\
  6.25e-03   &4.81e-01 &  0.05 &  7.59e-04 & 2.99    & 1.05e-01 & 1.43  \\
3.12e-03   &2.10e-01 &  1.20  & 9.50e-05 &  3.00     &2.93e-02 &  1.84   \\
 1.56e-03  & 1.81e-02 &  3.53 & 1.19e-05 &  3.00 &  7.53e-03 &  1.96 \\
7.81e-04   & 1.16e-03 &  3.97  & 1.48e-06 &  3.00  &  1.90e-03 & 1.99   \\
3.90e-04  & 7.24e-05&  4.00  & 1.85e-07 &  3.00  & 4.75e-04 & 2.00  \\
  1.95e-04 &4.53e-06 &  4.00  & 2.32e-08 &  3.00   & 1.19e-04 & 2.00   \\
   9.76e-05  &2.83e-07 &   4.00 & 2.90e-09 &  3.00  & 2.97e-05 & 2.00   \\
\hline
\end{tabular}

\begin{tabular}{|c|cc|cc|cc|}                            
\hline
& \multicolumn{6}{c|}{$u(x)= e^x$}  \\
\hline
$h$ & \multicolumn{2}{c|} {$\epsilon=10^{-6}$} & \multicolumn{2}{c|} {$\epsilon=h$}  & \multicolumn{2}{c|}{$\epsilon=h^2$}
 \\
& $C_0-\omega_0$ & rate & $C_0-\omega_0$ & rate& $C_0-\omega_0$ & rate\\
\hline
5.00e-02   &    3.16e-02 &           &  1.40e-03 &        & 1.52e-02 &        \\
2.50e-02  & 1.42e-02 &   1.16  &3.32e-04 &   2.08     & 6.96e-03 &  1.13  \\
1.25e-02   &6.64e-03 &  1.10  &  8.06e-05 & 2.04  & 3.30e-03 & 1.07  \\
  6.25e-03   &3.15e-03 &  1.07 &  1.98e-05 & 2.02    & 1.61e-03 & 1.04  \\
3.12e-03   &1.44e-03 &  1.13  & 4.92e-06 &  2.01     &7.93e-04 &  1.02   \\
 1.56e-03  & 5.59e-04 &  1.37 & 1.22e-06 &  2.01 &  3.93e-04 &  1.01 \\
7.81e-04   & 1.49e-04 &  1.91  & 3.06e-07 &  2.00  &  1.96e-04 & 1.01   \\
3.90e-04  & 2.59e-05&  2.52  & 7.64e-08 &  2.00  & 9.78e-05 & 1.00 \\
  1.95e-04 &4.59e-06 &  2.85  & 1.90e-08 &  2.00   & 4.89e-05 & 1.00   \\
   9.76e-05  &4.61e-07 &   2.96 & 4.77e-09 &  2.00  & 2.44e-05 & 1.00   \\
\hline
\end{tabular}
\end{center}
\caption{CWENO3. Distance and order of convergence of the optimal weights $C_0$ and the nonlinear weights $\omega_0$ , as a function of  the mesh width $h$, in the cell centered in $x=0$.}
\label{tab:CWENO3:Comega}
\end{table}

Furthermore, Tables \ref{tab:WENO3:Comega}
and \ref{tab:CWENO3:Comega} report the asymptotic behaviour of the nonlinear weights for the WENO3 and CWENO3 recontruction, comparing the case $\up\neq0$ and the case with a smooth extremum in the reconstruction stencil. The behaviours follow the results of  Remarks \ref{rem:referee:WENO} and \ref{rem:referee:CWENO}.

\begin{figure}
\begin{center}
\includegraphics[width=0.49\linewidth]{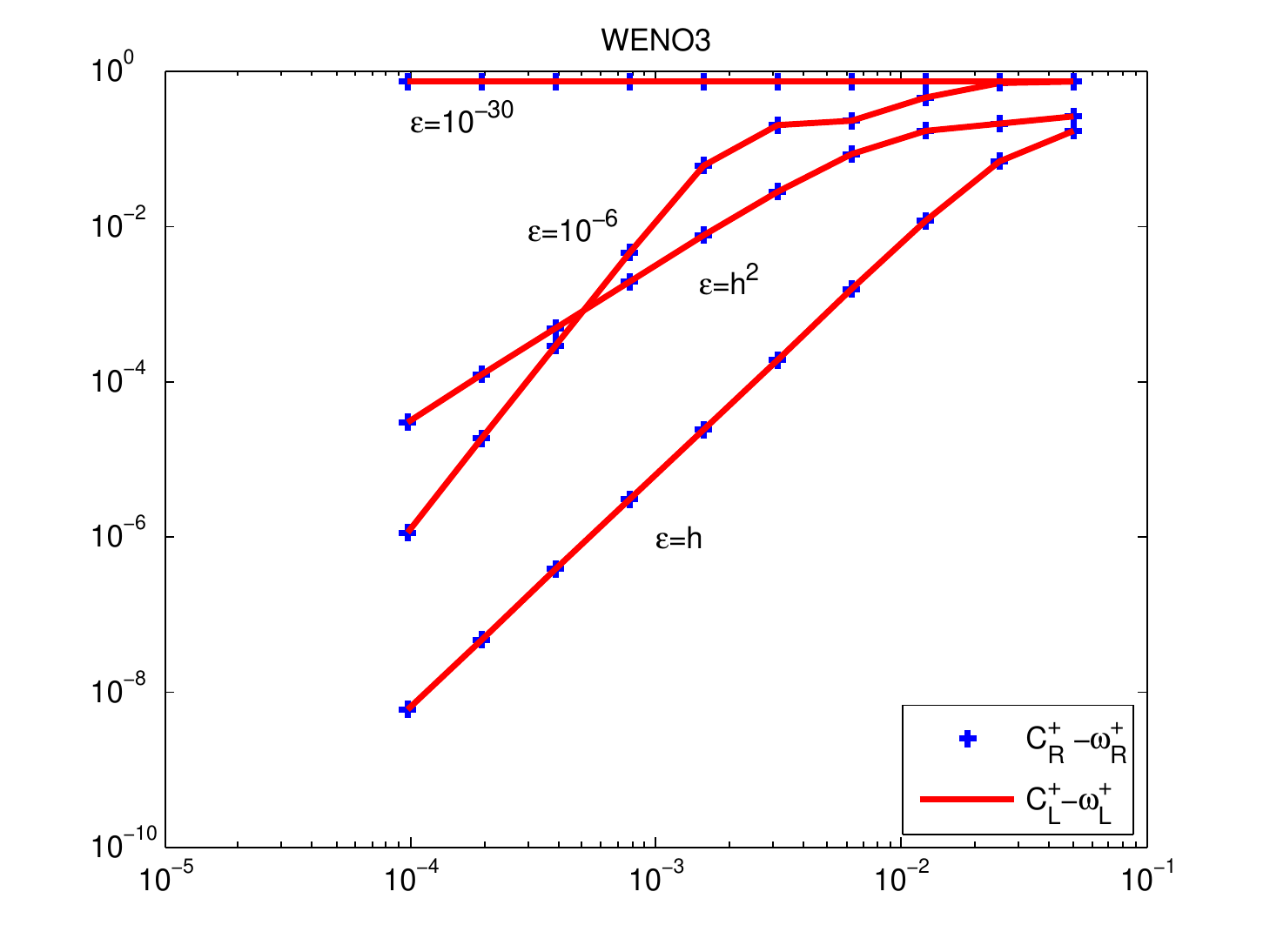}
\hfill
\includegraphics[width=0.49\linewidth]{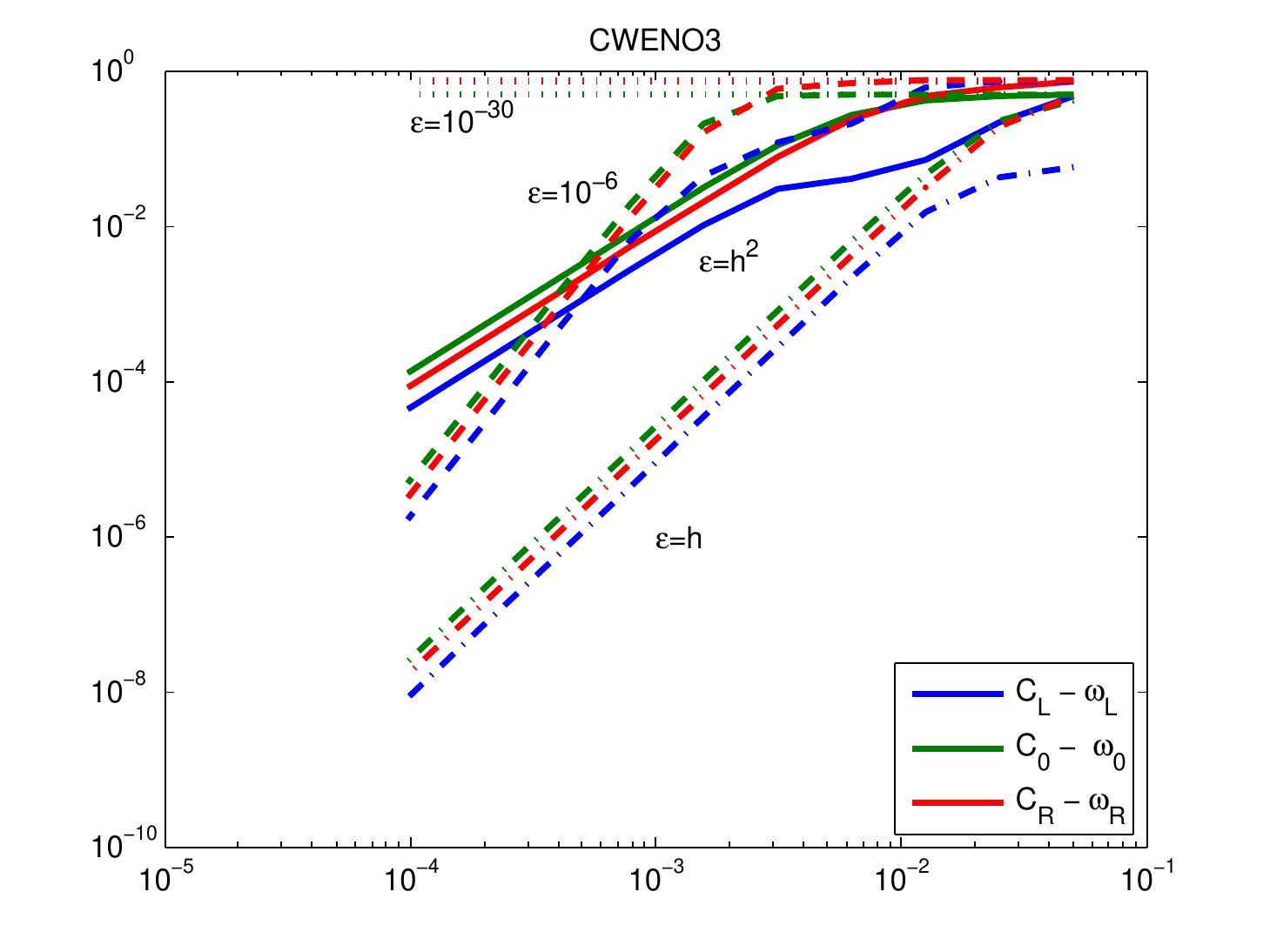}
\end{center}
\caption{Distance between the optimal weights $C_\lambda$ and the nonlinear weights $\omega_\lambda$ as a function of the mesh width $h$, for different choices of $\epsilon$. WENO3 (left) and CWENO3 (right)}
\label{fig:Comega}
\end{figure}

Finally, in order to investigate more deeply the behaviour of the reconstruction procedure, we denote by $\mathcal{R}(\ca{u}_{j-1}, \ca{u}_{j},  \ca{u}_{j+1})$ the map from the three cell averages to the reconstructed boundary value $u^-_{j+1/2}$ and by $\mathcal{P}_2(\ca{u}_{j-1}, \ca{u}_{j},  \ca{u}_{j+1})=P^{\text{OPT}}_j(x_{j+1/2})$ the reconstruction operator that employs the central parabola. The previous test confirmed that $\mathcal{R}(\ca{u}_{j-1}, \ca{u}_{j},  \ca{u}_{j+1})\to\mathcal{P}_2(\ca{u}_{j-1}, \ca{u}_{j},  \ca{u}_{j+1})$ for $h\to0$. In Figure \ref{fig:derivateDellaRic} we test the convergence $\nabla\mathcal{R}\to\nabla\mathcal{P}_2$. The plots show that for $h$-dependent $\epsilon$, $\partial\mathcal{R}/\partial{\ca{u}_k}$ converge quickly to $\partial \mathcal{P}_2/\partial{\ca{u}_k}$ for $k=j-1,j,j+1$. At the opposite hand, for $\epsilon=10^{-30}$ we do not observe such a convergence (note  that $\partial \mathcal{R}/\partial{\ca{u}_{k}}\not\to\partial \mathcal{P}_2/\partial{\ca{u}_{k}}$, $k=j-1,j,j+1$) and $\epsilon=10^{-6}$ shows an hybrid behaviour that changes regime when $h$ falls below a threshold. (These tests were performed on uniform meshes for $u(x)=x^3 + \cos x$).

\begin{figure}
\begin{center}
\includegraphics[width=0.49\linewidth]{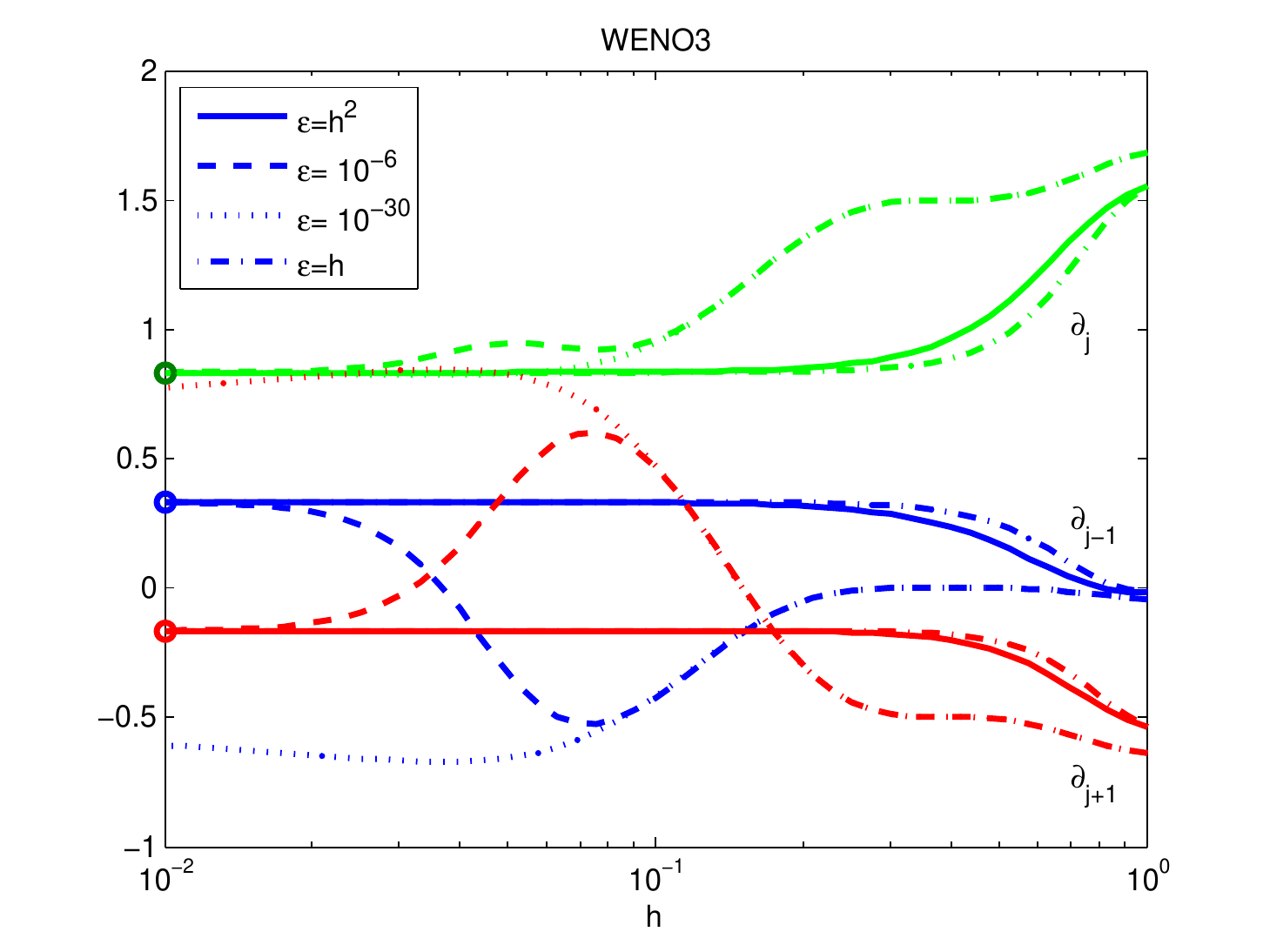}
\hfill
\includegraphics[width=0.49\linewidth]{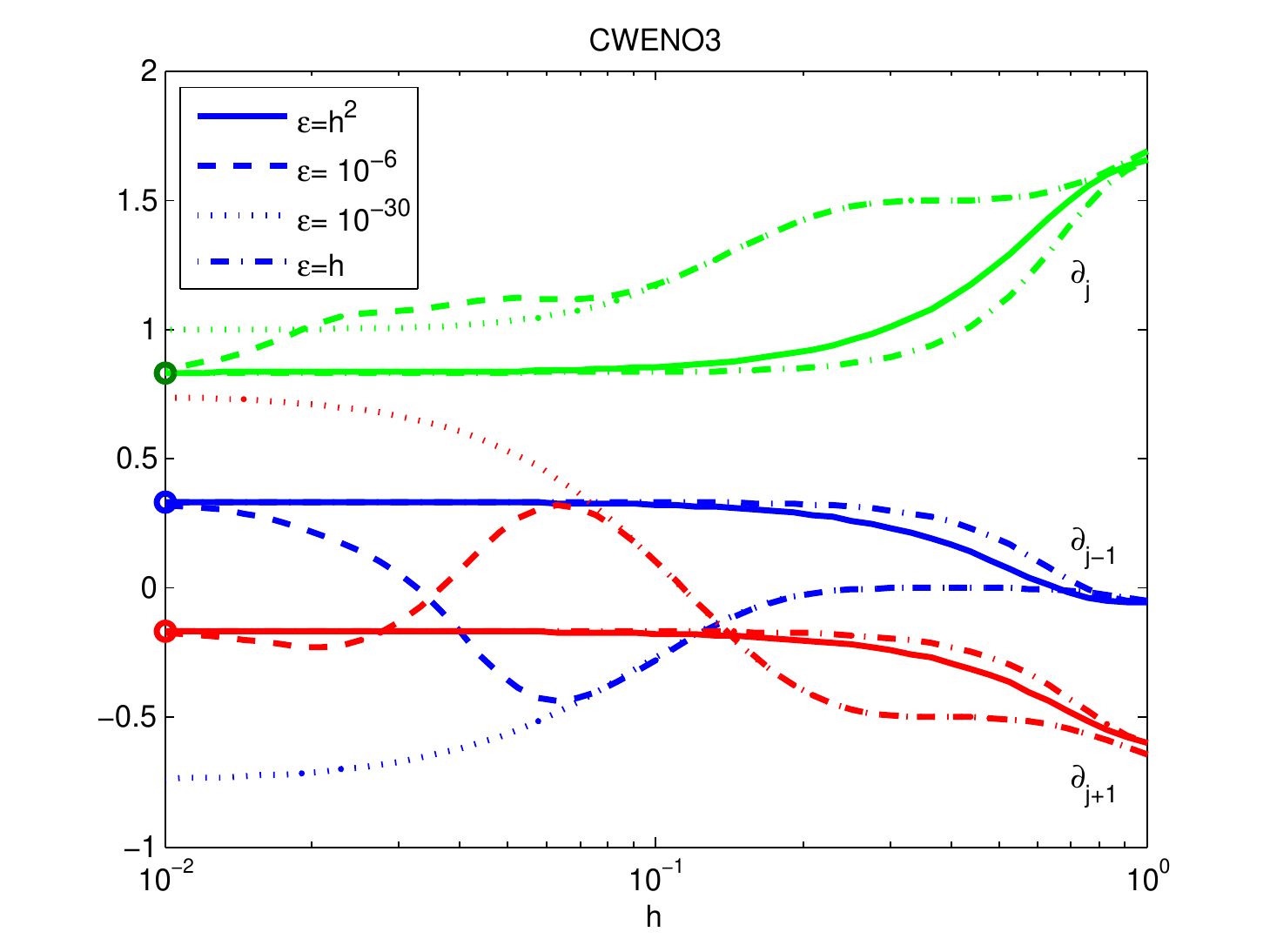}
\end{center}
\caption{Difference between the components of  $\nabla\mathcal{R}(\ca{u}_{j-1}, \ca{u}_{j},  \ca{u}_{j+1})$ and of $\nabla\mathcal{P}_2(\ca{u}_{j-1}, \ca{u}_{j},  \ca{u}_{j+1})$ as a function of the uniform mesh size $h$ for different choices of $\epsilon$. WENO3 (left) and CWENO3 (right). }
\label{fig:derivateDellaRic}
\end{figure}

\subsection{Numerical derivative and linear transport} 
\label{ssec:numerder}
In this set of tests we investigate the effects of the choice of $\epsilon$ 
 in a numerical scheme for the linear transport equation. 
Recall that, for $u_t+u_x=0$, when using upwind numerical fluxes, the semidiscrete scheme  \eqref{e:semischeme} boils down to 
\[
\frac{\D}{\D t}\ca{U}_j= - \frac{1}{h_j}\left( U^-_{j+1/2}- U^-_{j-1/2}\right) .
\]

This system of ODE is discretized with the third order, three stages SSPRK, see \cite{Gottlieb:2001}.
For these tests, a non uniform grid obtained repeating groups of four cells of size $ h, \, h/2, \, h/4, \, h/4 $ was generated.
We point out that cells of size $h$ have a neighbourhood of type  $\alpha,\beta,\gamma,\delta = 1/4, \, 1/4, \, 1/2, \, 1/4,$
 cells of size $h/2$ have  a neighbourhood with  $\alpha,\beta,\gamma,\delta = 1/2, \, 2, \, 1/2, \, 1/2,$ while
 cells of size $h/4$ have   $\alpha,\beta,\gamma,\delta = 4, \, 2, \, 1, \, 4$ or
 $\alpha,\beta,\gamma,\delta = 2, \, 1, \, 4, \, 2.$

We recall that, if the numerical flux function $\mathcal{F}$ appearing in \eqref{e:fluxes} is consistent and Lipschitz continuous, the reconstructions are third order accurate and the ODEs
\eqref{e:semischeme} are discretized using a third order accurate
method, then the numerical finite volume scheme is third order accurate in time as well in space, see \cite[\S 17]{LeVeque99}.

These tests compute also  the spatial discretization error 
\[
\big( u(x_{j+1/2})- u(x_{j-1/2})\big)/h_j -
\big( U^-_{j+1/2}- U^-_{j-1/2}\big)/h_j,
\]
which is the finite volume error analogue of the finite difference truncation error
\(
u^\prime(x_j) - \big( U^-_{j+1/2}- U^-_{j-1/2}\big)/h_j
\)
studied in \cite{Arandiga}.

We integrate $u_t+u_x=0$ on the domain $[0,1]$ until $t=1$, with periodic boundary conditions and the smooth initial datum $u_0(x)=\sin(2\pi x -\sin(2\pi x)/2\pi)$.
The maximum norm  error for the numerical derivative of $u_0(x)$  on this grid and the 1-norm  error at final time in the linear transport test were recorded.
The results for $ \epsilon= h^2$ are shown in Table \ref{tab:stupida:WENO3} (WENO3) and Table \ref{tab:stupida:CWENO} (CWENO3). Uniform  and quasi-regular grids  showed the expected convergence rates
 (\eqref{eq:gamma} and Remarks \ref{rem:uniform:WENO} and  \ref{rem:uniform:CWENO})  in both the spatial discretization error and the linear transport test.
Random and $\alpha\beta\gamma\delta$-grids show irregular and degraded convergence rates for the spatial discretization error, but the  order of convergence is $3$ in the linear transport test.

\begin{table}
\begin{center}
Maximum norm  error on the numerical derivative on $u_0(x)$
\vspace{5mm}
\begin{tabular}{|r|ll|ll|ll|ll|}
\hline
$N$ &  \multicolumn{2}{c|}{uniform} & \multicolumn{2}{c|}{quasi-unif} & \multicolumn{2}{c|}{random}& \multicolumn{2}{c|}{[1,1/2,1/4,1/4,1]}\\
& error & rate & error & rate& error & rate & error & rate \\
\hline
   20 &  6.30e-01  &      &  6.20e-01 &       &  7.20e-01 &       &  3.90e-04 &       \\
   40 &  3.06e-01  & 1.04 &  2.58e-01 &  1.26 &  2.83e-01 &  1.34 &  1.56e-04 &  1.32 \\
   80 & 5.46e-02   & 2.49 &  5.93e-02 &  2.12 &  5.21e-02 &  2.44 &  4.27e-05 &  1.87 \\
   160 &  7.52e-03  & 2.86  &  9.00e-03 &  2.72 &  8.37e-03 &  2.63 &  1.09e-05 &  1.97 \\
   320 &   9.78e-04 & 2.94  &  1.18e-03 &  2.92 &  1.34e-03 &  2.63 &  2.74e-06 &  1.99 \\
   640 &  1.23e-04 &  2.99  &  1.47e-04 &  3.00 &  4.03e-04 &  1.73 &  6.86e-07 &  2.00 \\
   1280 &  1.54e-05 & 3.00  &  1.83e-05 &  3.00 &  9.11e-05 &  2.14 &  1.71e-07 &  2.00 \\
   2560 &  1.92e-06 & 3.00  &  2.27e-06 &  3.00 &  2.34e-05 &  1.95 &  4.29e-08 &  2.00 \\
\hline
\end{tabular}
\end{center}
\begin{center}
1-norm  error on the linear transport at $t=1$
\begin{tabular}{|r|ll|ll|ll|ll|}
\hline
$N$ &  \multicolumn{2}{c|}{uniform} & \multicolumn{2}{c|}{quasi-unif} & \multicolumn{2}{c|}{random}& \multicolumn{2}{c|}{[1,1/2,1/4,1/4,1]}\\
& error & rate & error & rate& error & rate & error & rate \\
\hline
20  & 8.20e-02 &   &9.69e-02 &       &  8.31e-02 &       &  4.10e-02 &       \\
40 & 2.75e-02&  1.57  &4.20e-02 &  1.20 &  2.76e-02 &  1.58 &  8.33e-03 &  2.30 \\
80 &4.95e-03 & 2.48  &9.47e-03 &  2.15 &  5.01e-03 &  2.46 &  1.25e-03 &  2.73 \\
160 & 7.35e-04 & 2.75  &1.56e-03 &  2.60 &  7.40e-04 &  2.75 &  1.61e-04 &  2.96 \\
320&  9.36e-05 &2.97  &2.10e-04 &  2.88 &  9.40e-05 &  2.97 &  1.94e-05 &  3.06 \\
640&1.14e-05 & 3.04  &2.57e-05 &  3.03 &  1.14e-05 &  3.03 &  2.38e-06 &  3.03 \\
1280 &1.41e-06 &3.01  &3.16e-06 &  3.02 &  1.41e-06 &  3.01 &  2.97e-07 &  3.00 \\
2560 &1.76e-07 & 3.00 &3.94e-07 &  3.00 &  1.76e-07 &  3.00 &  3.71e-08 &  3.00 \\
\hline
\end{tabular}
\end{center}
\caption{ Discrete  maximum norm  error on numerical derivative (top) and discrete 1-norm  error in 
linear transport equation with WENO3. $\epsilon=h^2$,  $u_0(x)=\sin\left(2\pi x-\sin(2\pi x)/2\pi\right).$}
\label{tab:stupida:WENO3}
\end{table}

\begin{table}
\begin{center}
Maximum norm   error on the numerical derivative  on $u_0(x)$
\vspace{5mm}
\begin{tabular}{|r|ll|ll|ll|ll|}
\hline
$N $ &  \multicolumn{2}{c|}{uniform} & \multicolumn{2}{c|}{quasi-unif} & \multicolumn{2}{c|}{random}& \multicolumn{2}{c|}{[1,1/2,1/4,1/4,1]}\\
& error & rate & error & rate& error & rate & error & rate \\
\hline 
   20 &  4.83e-01   &  &  8.77e-01 &       &  6.55e-01 &       &  5.13e-04 &       \\
   40 &  3.14e-01  & 0.62  &  2.62e-01 &  1.74 &  3.06e-01 &  1.09 &  1.64e-04 &  1.64 \\
   80 &   5.05e-02 & 2.64  &  5.48e-02 &  2.26 &  7.18e-02 &  2.09 &  4.32e-05 &  1.92 \\
   160 &  5.46e-03 & 3.21   &  6.59e-03 &  3.06 &  8.61e-03 &  3.06 &  1.09e-05 &  1.98 \\
   320 &  5.98e-04 & 3.19  &  7.58e-04 &  3.12 &  1.51e-03 &  2.51 &  2.74e-06 &  1.99 \\
   640  &  7.10e-05 & 3.07  &  8.83e-05 &  3.10 &  2.77e-04 &  2.44 &  6.86e-07 &  2.00 \\
   1280 &  8.73e-06 & 3.02  &  1.06e-05 &  3.05 &  6.43e-05 &  2.10 &  1.72e-07 &  2.00 \\
   2560 &  1.09e-06 & 3.01 &  1.31e-06 &  3.02 &  1.61e-05 &  1.99 &  4.29e-08 &  2.00 \\
\hline
\end{tabular}
\end{center}
\begin{center}
 1-norm  error on the linear transport at $t=1$
\begin{tabular}{|r|ll|ll|ll|ll|}
\hline
$N$  &  \multicolumn{2}{c|}{uniform}  & \multicolumn{2}{c|}{quasi-unif} & \multicolumn{2}{c|}{random}& \multicolumn{2}{c|}{[1,1/2,1/4,1/4,1]}\\
& error & rate & error & rate& error & rate & error & rate \\
\hline
20  &  8.22e-02 &   & 9.91e-02 &       &  8.27e-02 &       &  3.90e-02 &       \\
40 &2.40e-02 & 1.78  &4.02e-02 &  1.30 &  2.41e-02 &  1.77 &  6.55e-03 &  2.57 \\
80 & 3.57e-03 & 2.75  &7.68e-03 &  2.38 &  3.60e-03 &   2.74 &  8.54e-04 &  2.94 \\
160 & 4.57e-04 & 2.97  &1.05e-03 &  2.87 &  4.64e-04 &  2.95 &  9.91e-05 &  3.11 \\
320& 5.36e-05&  3.09 &1.25e-04 &  3.06 &  5.45e-05 &  3.09 &  1.07e-05 &  3.20 \\
640&6.35e-06 &  3.08  &1.45e-05 &  3.10 &  6.49e-06 &  3.07 &  1.25e-06 &  3.11 \\
1280 & 7.80e-07 & 3.02  &1.76e-06 &  3.04 &  7.97e-07 &  3.02 &  1.52e-07 &  3.03 \\
2560  &9.72e-08 & 3.00  &2.18e-07 &  3.01  &  9.94e-08 &  3.00 &  1.89e-08 &  3.01 \\
\hline
\end{tabular}
\end{center}
\caption{ Discrete maximum norm   error on numerical derivative (top) and discrete  1-norm error in 
linear transport equation with CWENO3. $\epsilon=h^2$,  $u_0(x)=\sin\left(2\pi x-\sin(2\pi x)/2\pi\right).$}
\label{tab:stupida:CWENO}
\end{table}

Having shown that random and $\alpha\beta\gamma\delta$-grids are the most troublesome, but that in the linear transport test the theoretical order of convergence is easily reached with an $h$-dependent $\epsilon$, next we compare the different combinations of WENO3 and CWENO3 with $\epsilon=h^2$ and $\epsilon=h$ on linear transport test for smooth and discontinuous data.

\begin{figure}
\begin{center}
\includegraphics[width=0.49 \linewidth]{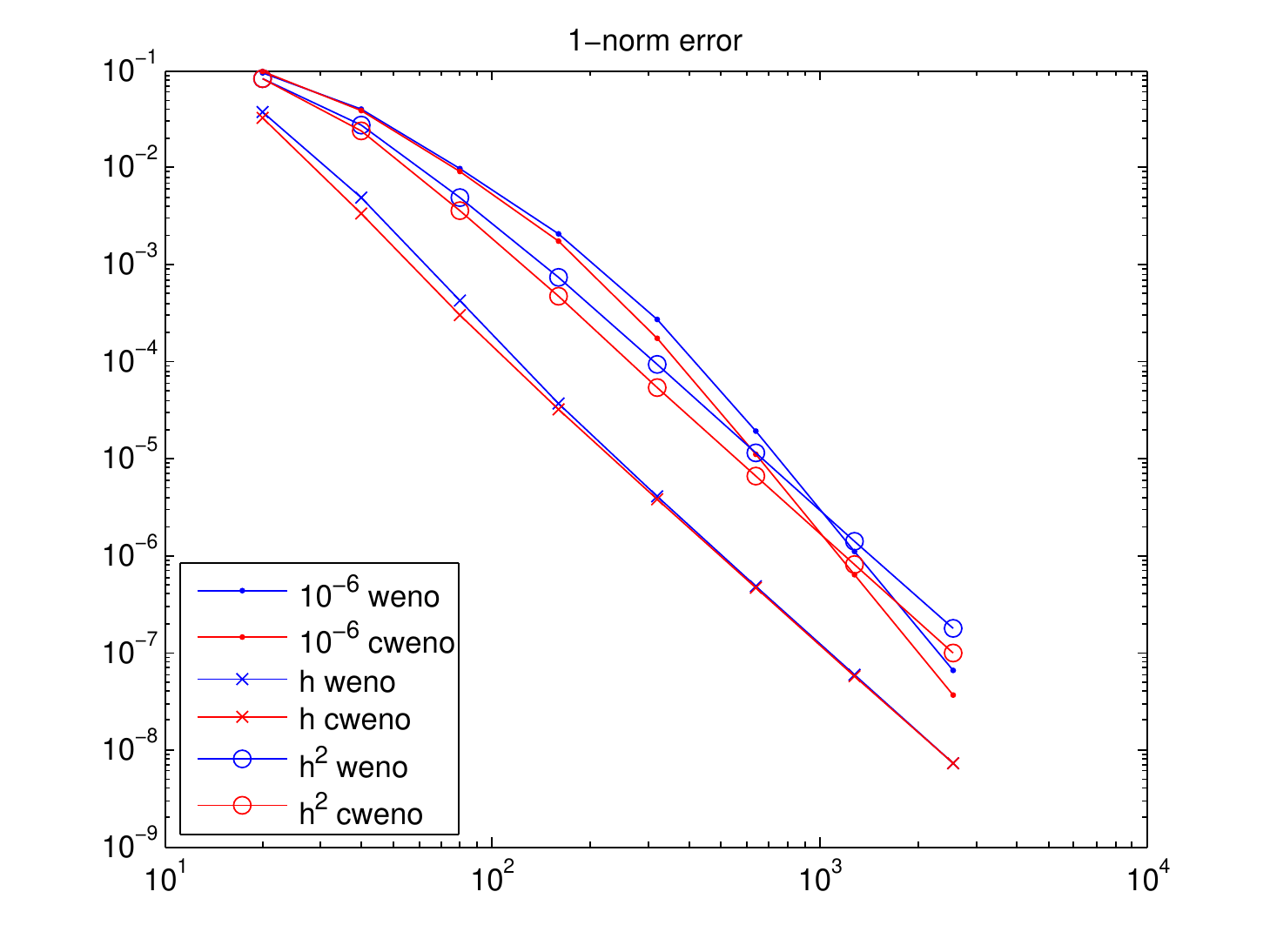}
\hfill
\includegraphics[width=0.49 \linewidth]{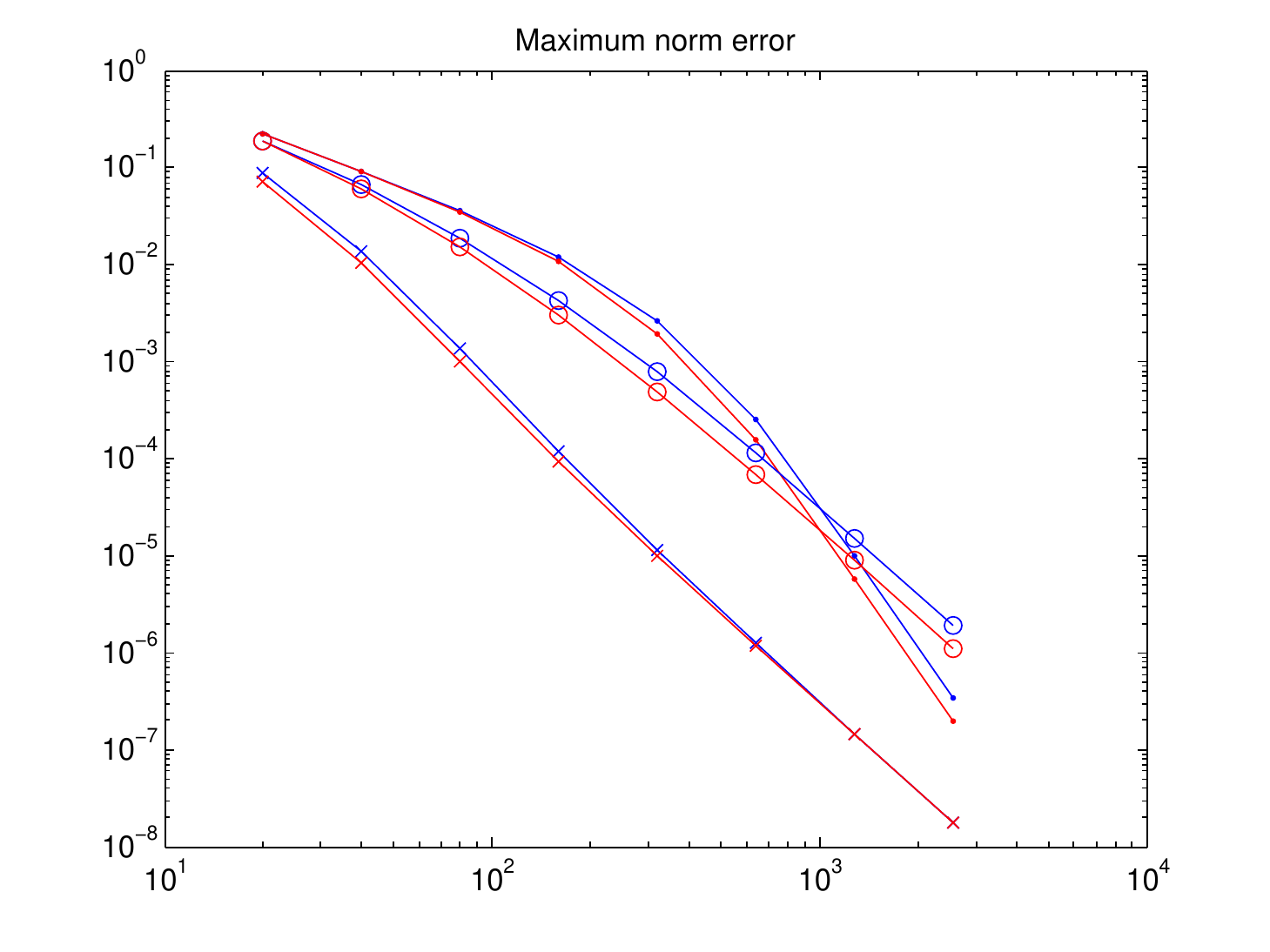}
\end{center}
\caption{Linear transport of smooth data on random grids. Discrete 1-norm  error (left) and discrete  maximum norm (right).}
\label{fig:lintra:smooth}
\end{figure}

 The error at final time in both the 1-norm  and the  maximum norm  are reported in Figure \ref{fig:lintra:smooth}. On all grid types (only random ones are shown) and for both norms, the choice $\epsilon=10^{-6}$ gives the biggest errors, $\epsilon=h^2$ is slightly better, while $\epsilon=h$ yields errors that are lower by about a factor of $2$. In all cases, the CWENO3 reconstruction yields slightly lower errors than the WENO3 reconstruction.

\begin{figure}
\begin{center}
\includegraphics[width=0.8\linewidth,height=0.3\linewidth]{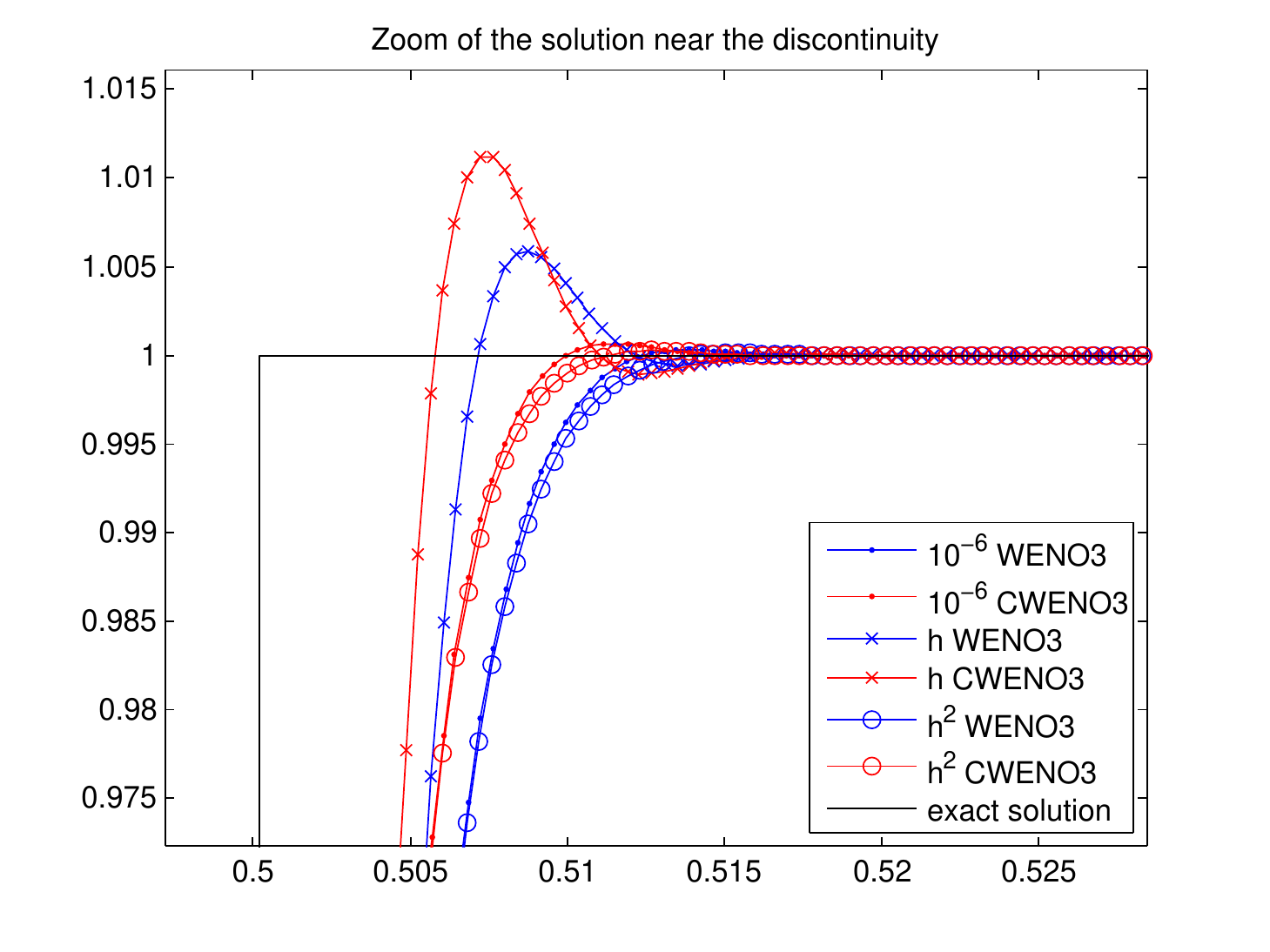}
\\
\includegraphics[width=0.49 \linewidth]{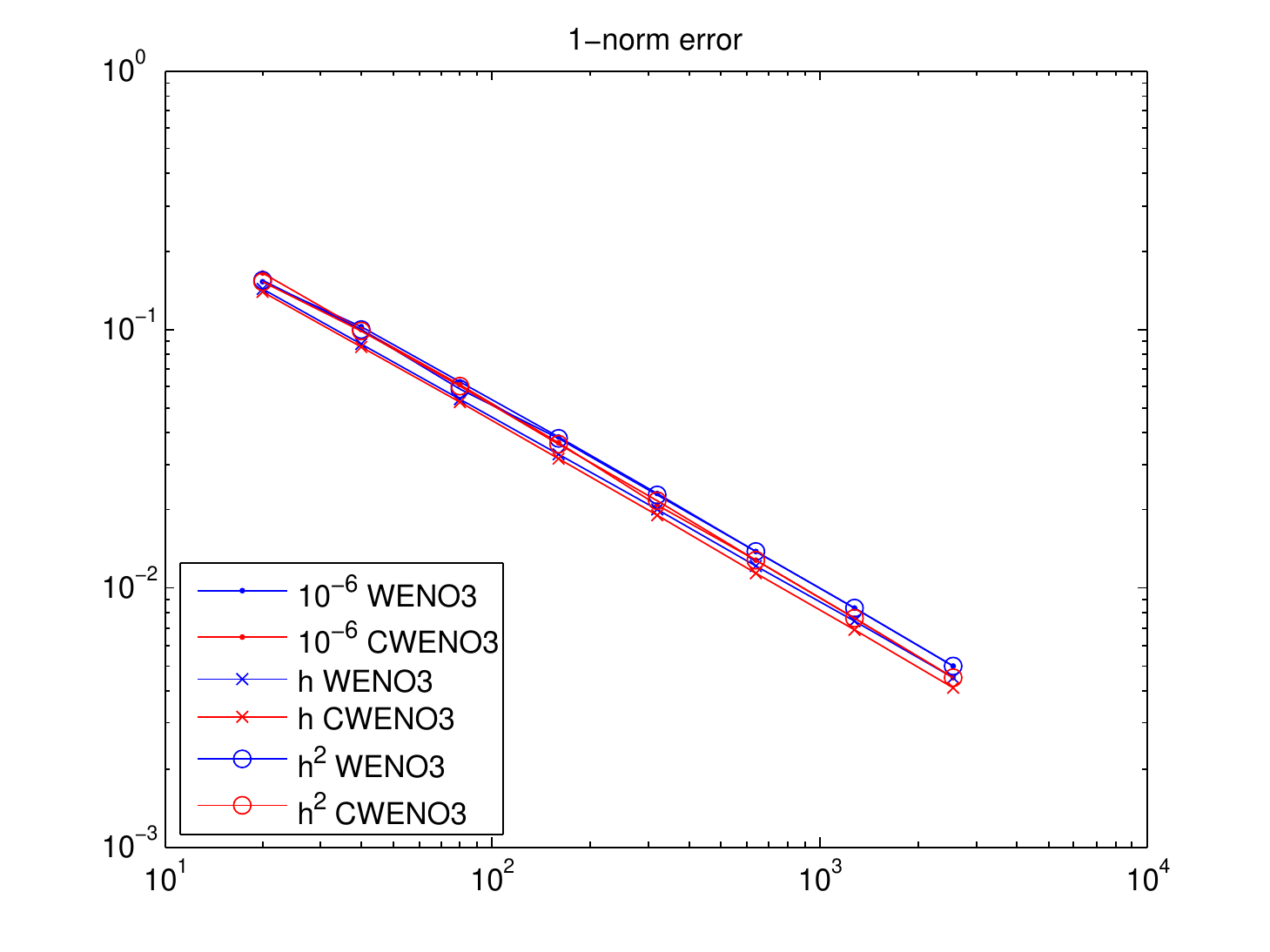}
\hfill
\includegraphics[width=0.49 \linewidth]{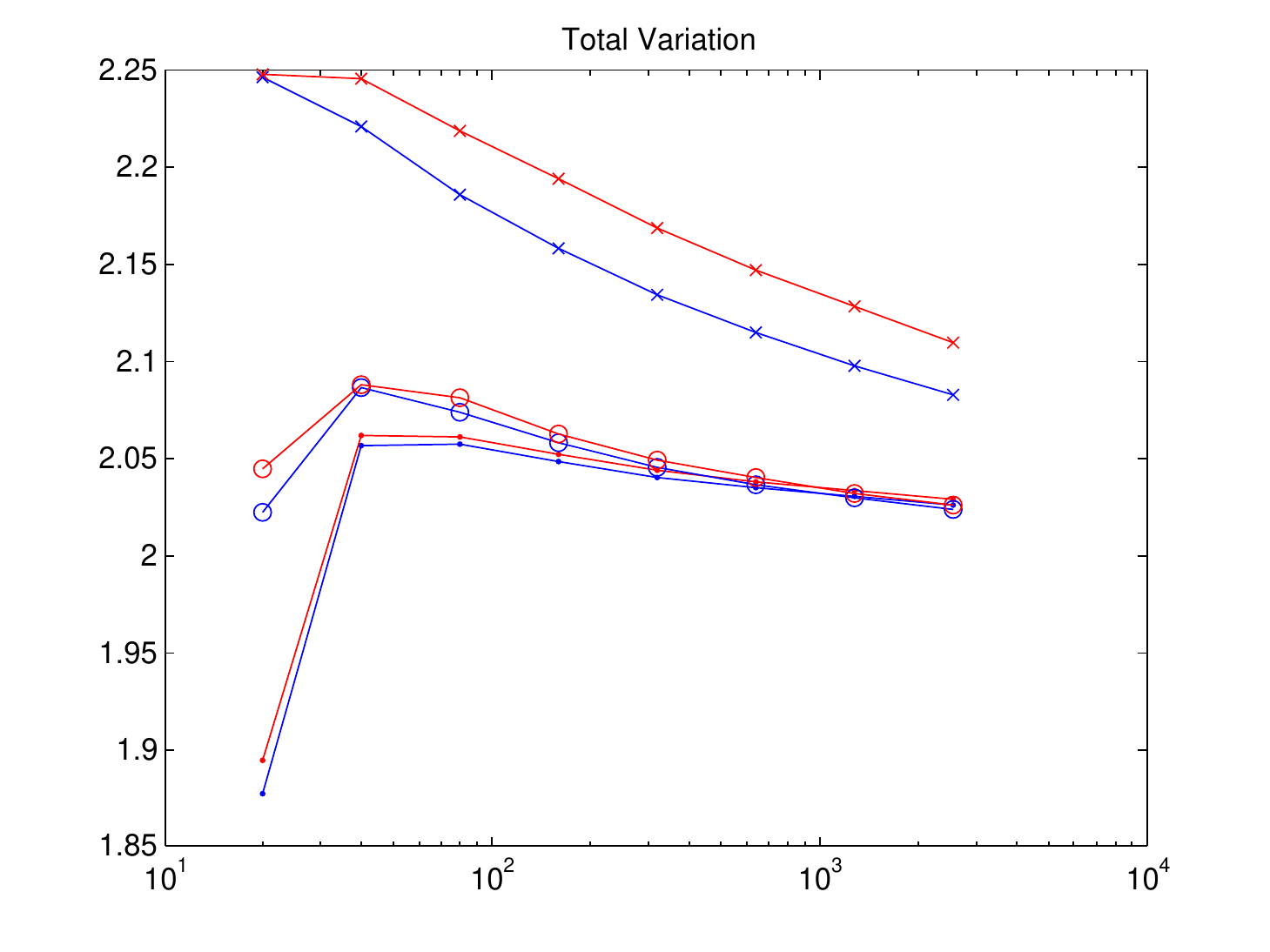}
\end{center}
\caption{ Linear transport of discontinuous data on random grids. Zoom of the solution close to the discontinuity (top),
the discrete 1-norm and the total variation, respectively (bottom).}
\label{fig:lintra:disc}
\end{figure}

Next we repeat the previous test with the square wave initial datum $u_0(x)=\chi_{[1/2,1]}(x)$, where $\chi$ denotes the characteristic function. At final time we computed the 1-norm of the error and the total variation: the results are reported in Figure \ref{fig:lintra:disc}. The errors in  1-norm  are much closer to each other than in the smooth case, with $\epsilon=h$ still being slightly better than the other choices. The test on the total variation shows that the increased resolution in the smooth case is obtained when the reconstructions stay closer to the central one and in fact the choices giving the lower errors in the previous test (i.e. CWENO3 and $\epsilon=h$) produce more total variations than the other ones. In any case the total variation is diminishing under grid refinement.

\begin{figure}
\begin{center}
\includegraphics[width=0.49\linewidth]{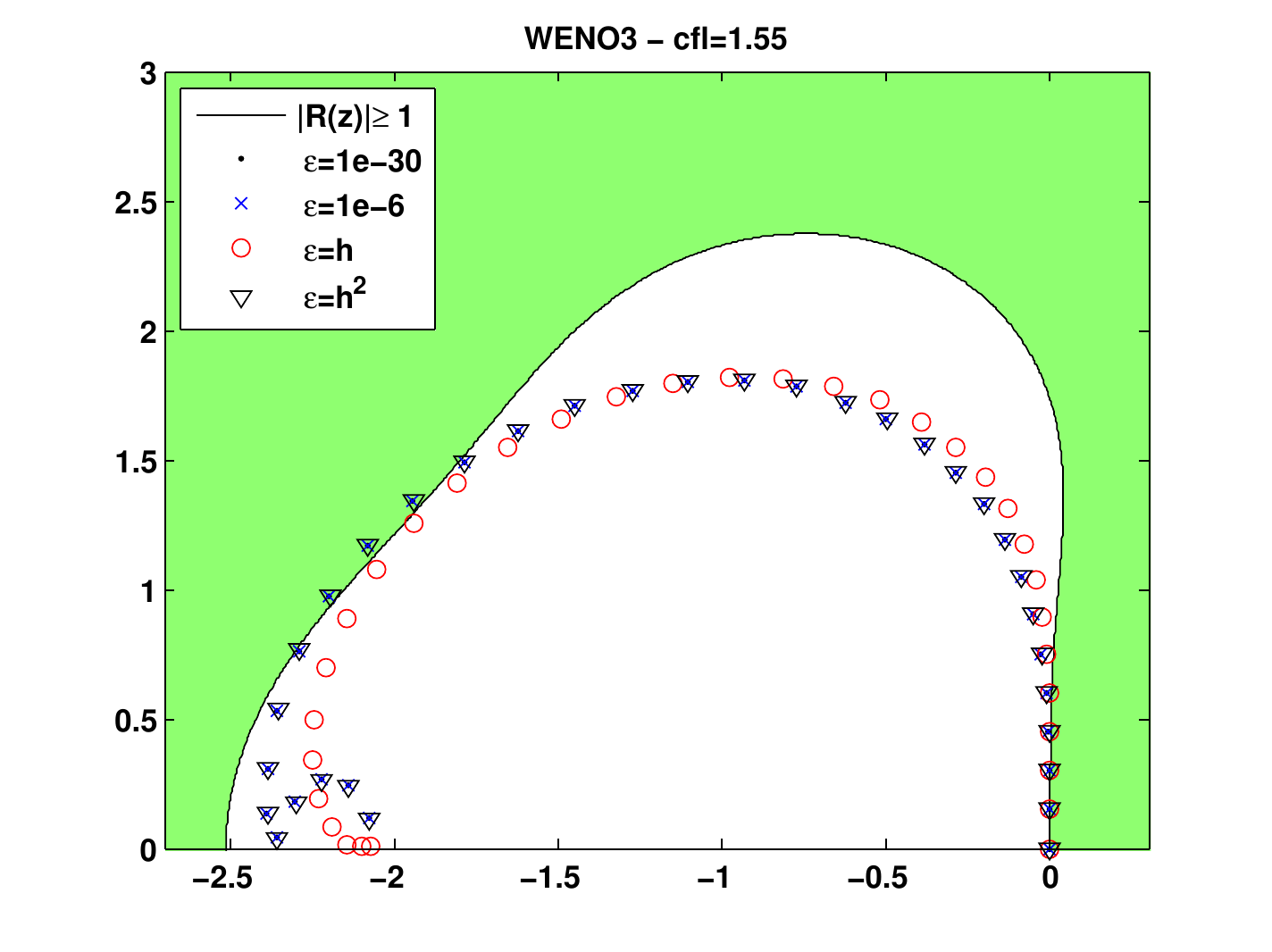}
\hfill
\includegraphics[width=0.49\linewidth]{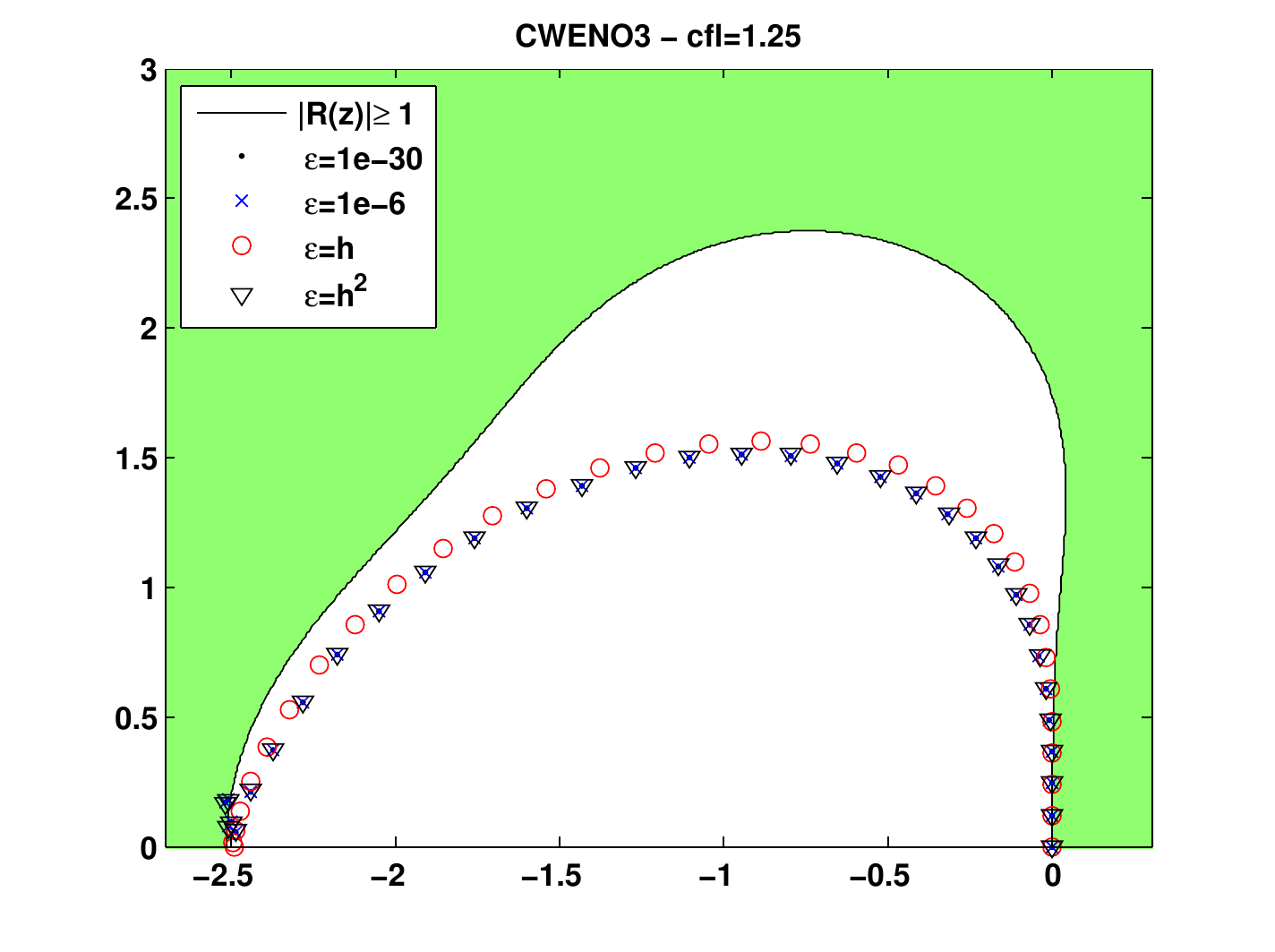}
\end{center}
\caption{Stability region of SSP-RK3 and spectrum of first derivative operator computed with upwind flux and WENO3 reconstructions (left) and CWENO3 reconstruction (right).}
\label{fig:stability}
\end{figure}

\subsection{Stability}
In Figure \ref{fig:stability} we compare the stability region of the third order SSPRK used for time advancement and the spectrum of the operator that computes numerically the first order derivative in the approximation of the linear transport equation $u_t+u_x=0$. The spectrum represented in the figure is the spectrum of the linearization of the nonlinear operator in the Fourier basis, which represents an analogue of the classical Von Neumann analysis that is applied to linear schemes.
In particular, on a grid of $65$ cells, we compute column-wise a $65\times65$ matrix $M$ as follows: for each column, we set up initial data coinciding with one of the real-valued Fourier basis with frequency between 0 and 32, compute the boundary extrapolated data, the numerical fluxes, the vector of differences $u^-_{j+1/2}-u^-_{j-1/2}$ and decompose this latter quantity along the real-valued Fourier basis. The eigenvalues of $M$ are shown in the picture, for WENO3 (left) and CWENO3 (right) and different choices of $\epsilon$. Linear stability is achieved if all the eigenvalues stay within the stability region of the Runge-Kutta scheme, i.e. in the white region of the plot. By symmetry, only the positive imaginary half-plane is shown. The CFL number was chosen so that the eigenvalues are very close to the boundary of the absolute stability region, showing that, for both WENO3 and CWENO3,  the choice $\epsilon=h$ has a slight stability advantage over the other three.
We also observe that the spectrum obtained with the CWENO3 reconstruction is less elongated in the imaginary axis direction than the one obtained with WENO3, as indicated by the different location at which the spectrum touches the boundary of the stability region of the Runge-Kutta scheme.

\subsection{Nonlinear conservation and balance laws}
\paragraph{Interplay with h-adaptive schemes}
The use of non uniform meshes is very important for non-linear conservation laws, since in the area around shocks, the truncation error can not be better than $\mathcal{O}(h)$ irrespectively of the order of the scheme employed and thus grid refinement is the only available tool to reduce the computational error in that area. The results of this paper are relevant for moving mesh methods  (quasi-uniform grids) and for locally refined $\alpha\beta\gamma\delta$-grids. Here we present tests in the latter setting, using the third order generalization of the h-adaptive scheme presented in \cite{PS:entropy},  which was analized in \cite{CRS}. The scheme employs a single non-uniform mesh, built from an initial uniform one by locally and recursively splitting in two the troubled cells according to some error indicator. As in \cite{PS:entropy}, the numerical entropy production is employed as an error indicator, but we believe that this result are rather independent on the details of the adaptive strategy, since in any case the numerical scheme would have to deal with nearby cells whose size ratio is a power of two and are thus not quasi-uniform.

\begin{figure}
\begin{center}
\includegraphics[width=0.45\textwidth]{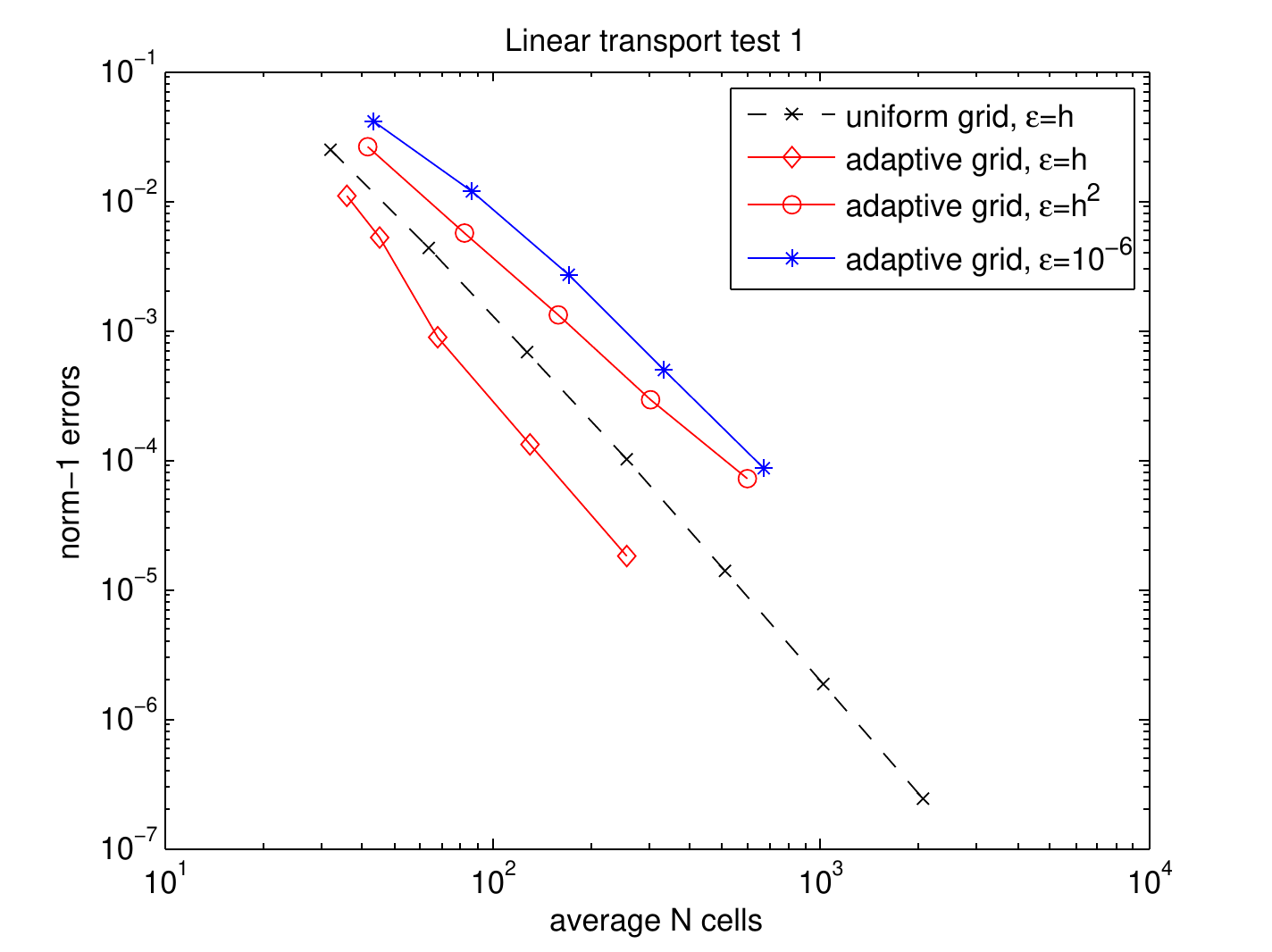}
\hfill
\includegraphics[width=0.45\textwidth]{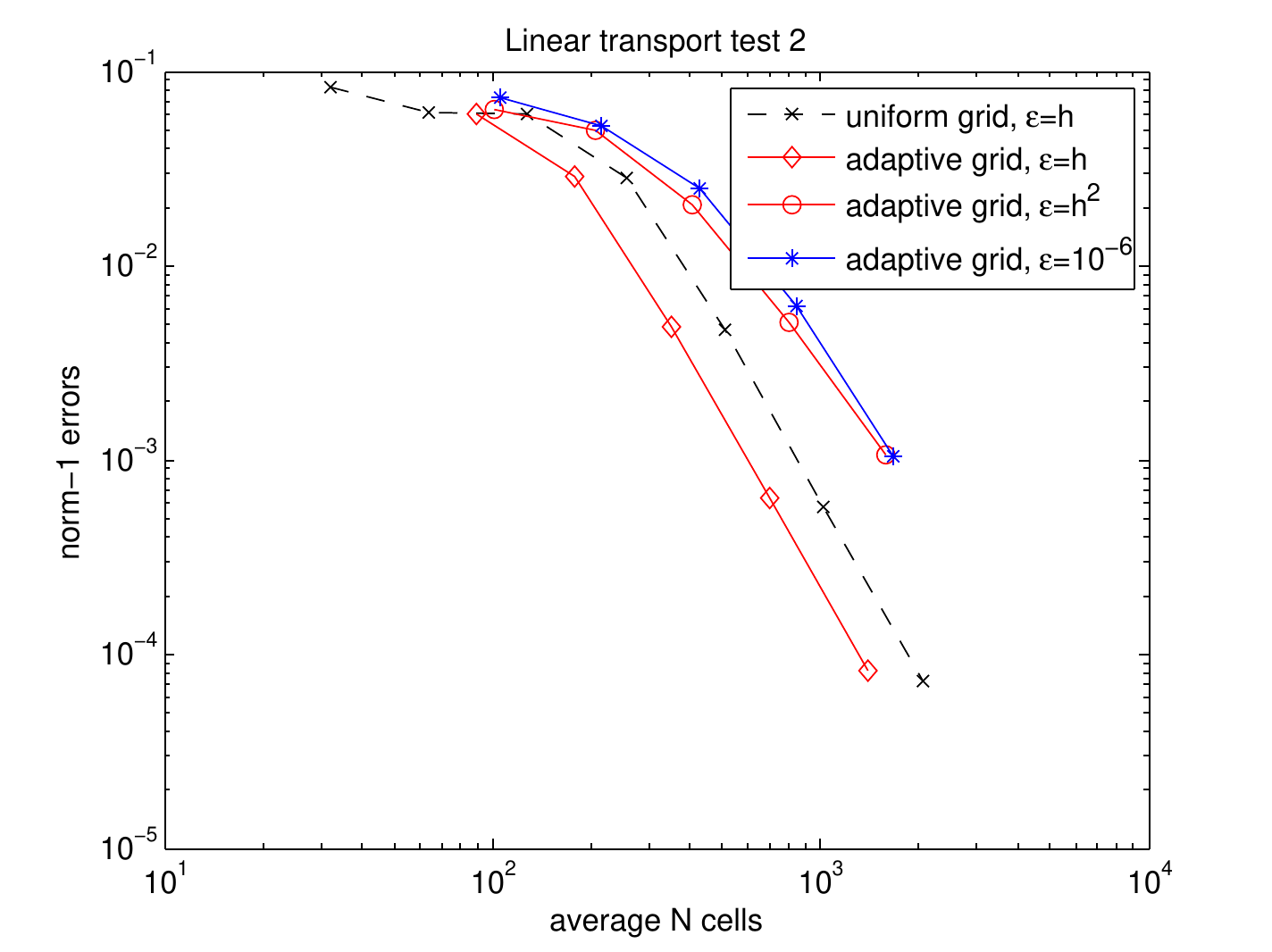}
\end{center}
\caption{Efficiency diagrams for $u_t+u_x=0$ with periodic boundary conditions and smooth data. The initial datum is slowly varying on the left and quite oscillating on the right.}
\label{fig:adap:lintra}
\end{figure}

In Figure \ref{fig:adap:lintra} we present the results obtained integrating the linear transport equation with periodic boundary conditions on $[0,1]$ and initial datum $u_1(x)=\sin(\pi x-\sin(\pi x)/\pi)$
(left) and $u_2(x)=\sin(\pi x)+0.25 \sin(15\pi x)e^{-20x^2}$ (right). The grid is adaptively refined during the numerical integration, using cells of three (four) sizes respectively, thus with maximum grid ratio of $1/8$ or $1/16$. The  discrete 1-norm error is plotted against the average number of cells that were used during the computation, which is a proxy of the computational effort of the scheme. Since the flow is smooth, in this test there are no theoretical reasons why non-uniform grids would outperform uniform ones. However, the figure shows that choosing an $h$-dependent $\epsilon$ yields consistently lower errors and that the choice $\epsilon=h$ clearly outperforms the other two.

\begin{figure}
\begin{center}
\includegraphics[width=0.45\textwidth]{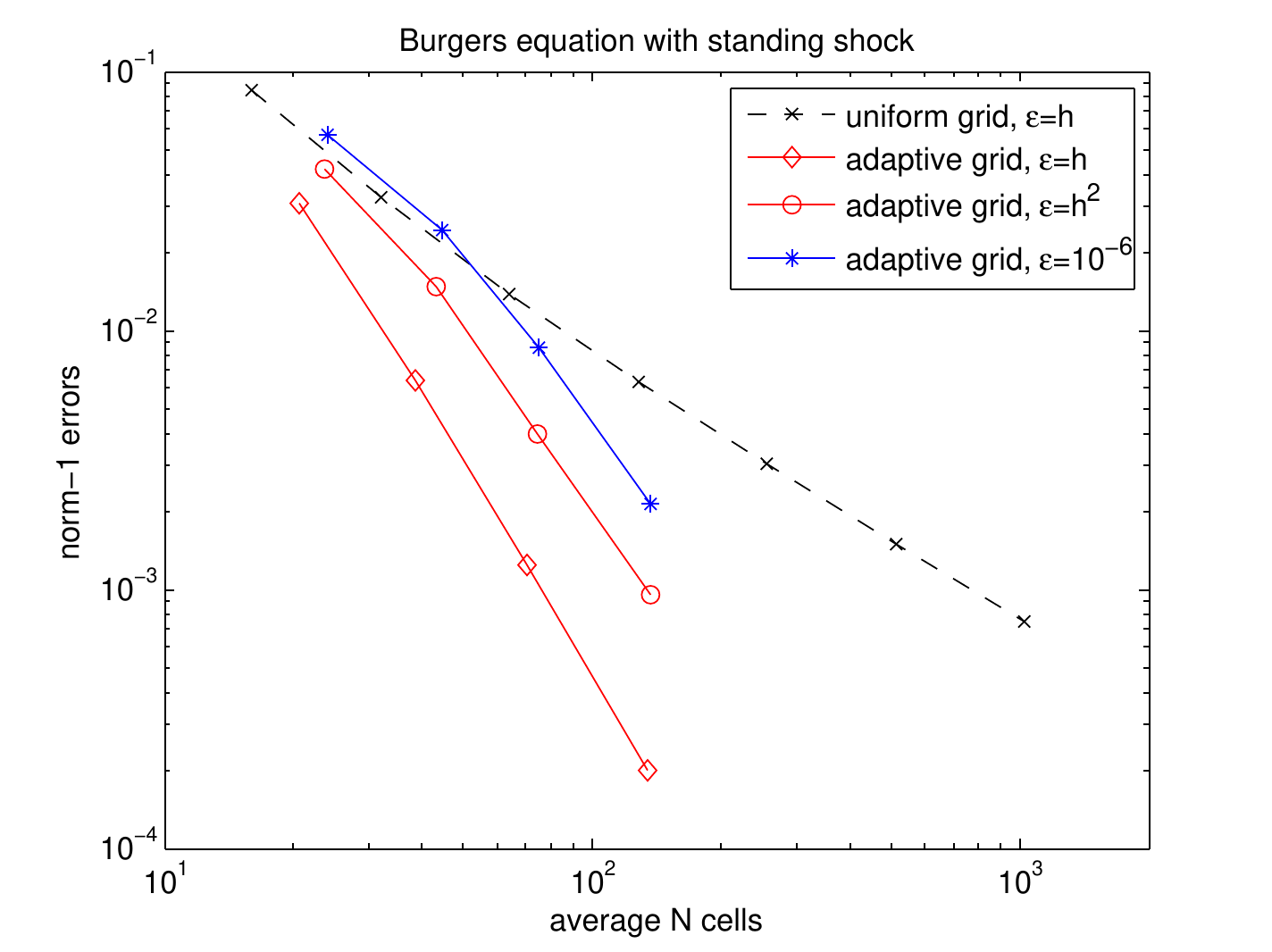}
\hfill
\includegraphics[width=0.45\textwidth]{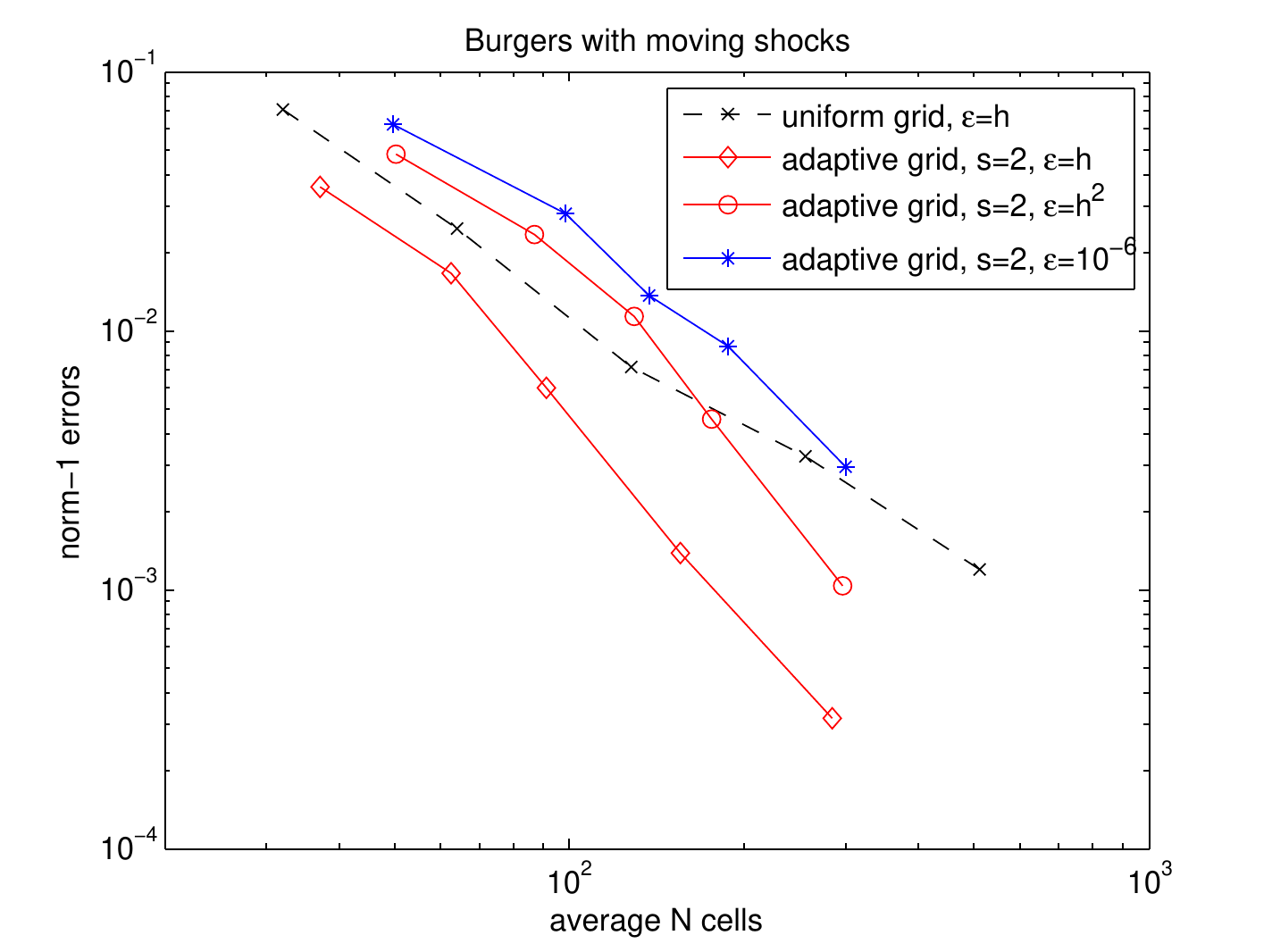}
\end{center}
\caption{Efficiency diagrams for the Burgers' equation. The initial datum on the left is $u_3$ and gives rise to a standing shock (final time $0.35$), while the one on the right is $u_4$, which is much more oscillating and gives rise to a complex solution structure with moving shocks (final time $0.45$).}
\label{fig:adap:burgers}
\end{figure}

In Figure \ref{fig:adap:burgers} we present the results obtained integrating the Burgers' equation $u_t+\tfrac12 (u^2)_x=0$ on the domain $[0,1]$ with the initial datum $u_3(x)=-\sin(\pi x)$ (left) and $u_4(x)=-\sin(\pi x)+0.2\sin(5.0\pi x)$ (right). The adaptive scheme clearly converges at a faster rate than the uniform grid one which is locked into first order behaviour by the presence of the shocks.
In the tests shown on the left, $u_3$ gives rise to a standing shock in the middle of the domain and we employed a coarse grid of $16\cdot2^k$ cells with $3+2k$ cell sizes (k=0\ldots3) as this is the choice that could recover third order convergence with respect to the average number of cells (see \cite{CRS}). In the test on the right, $u_4$ gives rise to a richer solution structure with two moving shocks converging towards the middle and four little waves where third order accuracy plays an important role. Here we employed $16\cdot2^k$ cells with $4+2k$ cell sizes (k=0,\ldots, 4).
In both cases, choosing an $h$-dependent $\epsilon$ yields consistently lower errors and the choice $\epsilon=h$ clearly outperforms the other two, especially in the more complex situation depicted on the right.

\begin{figure}
\begin{center}
\includegraphics[width=0.45\textwidth]{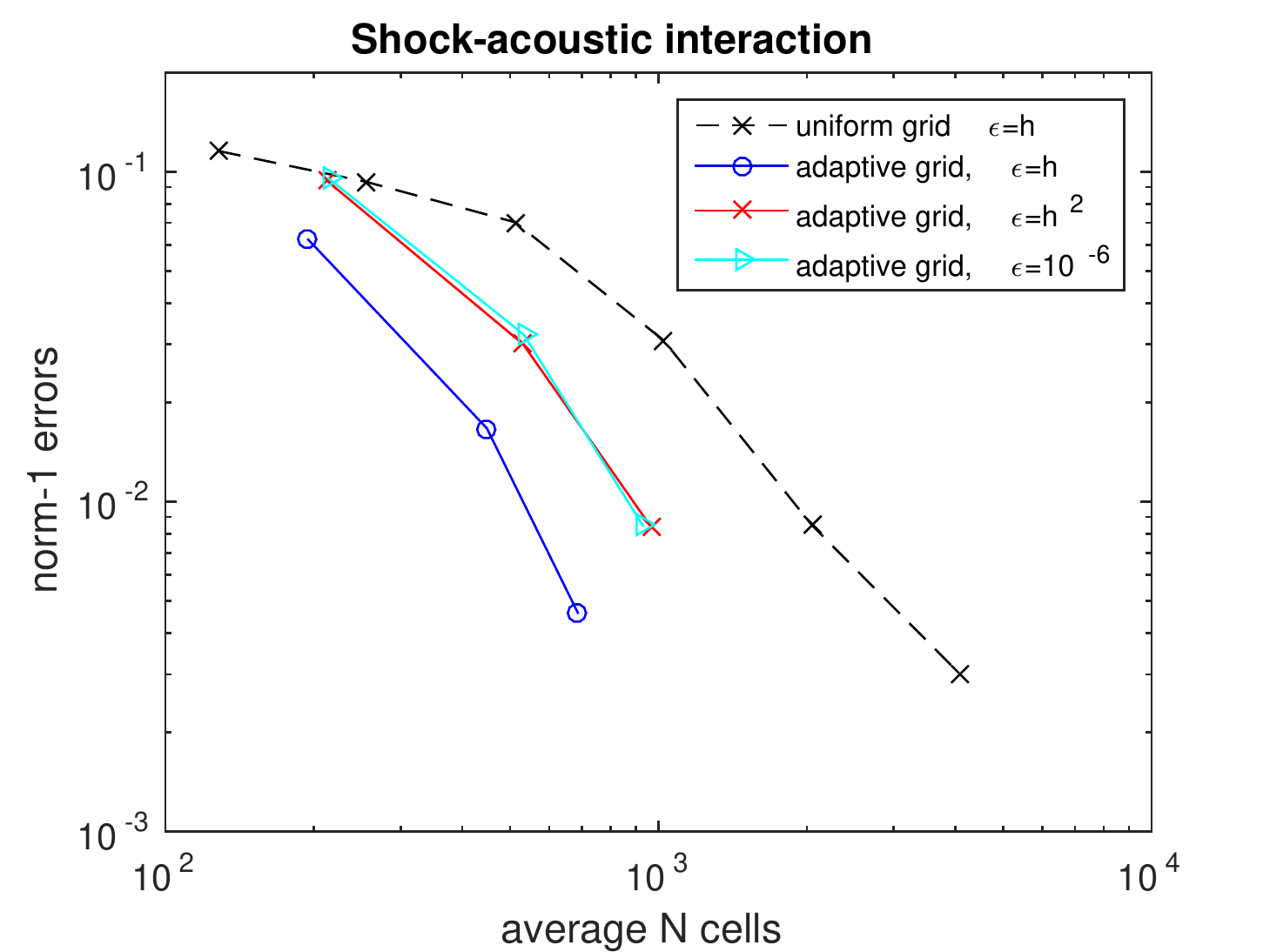}
\hfill
\includegraphics[width=0.45\textwidth]{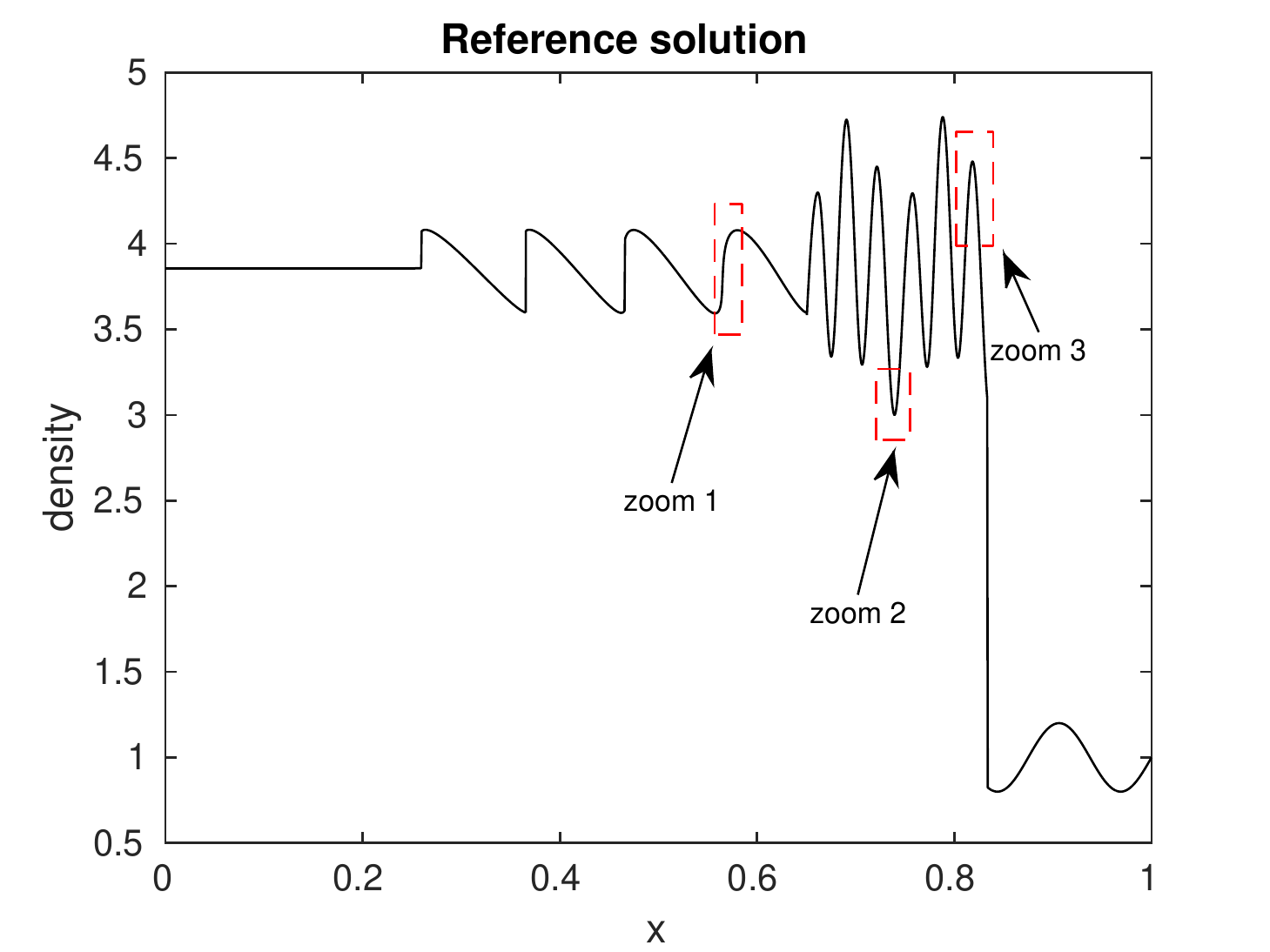}
\\[5mm]
\includegraphics[width=0.3\textwidth]{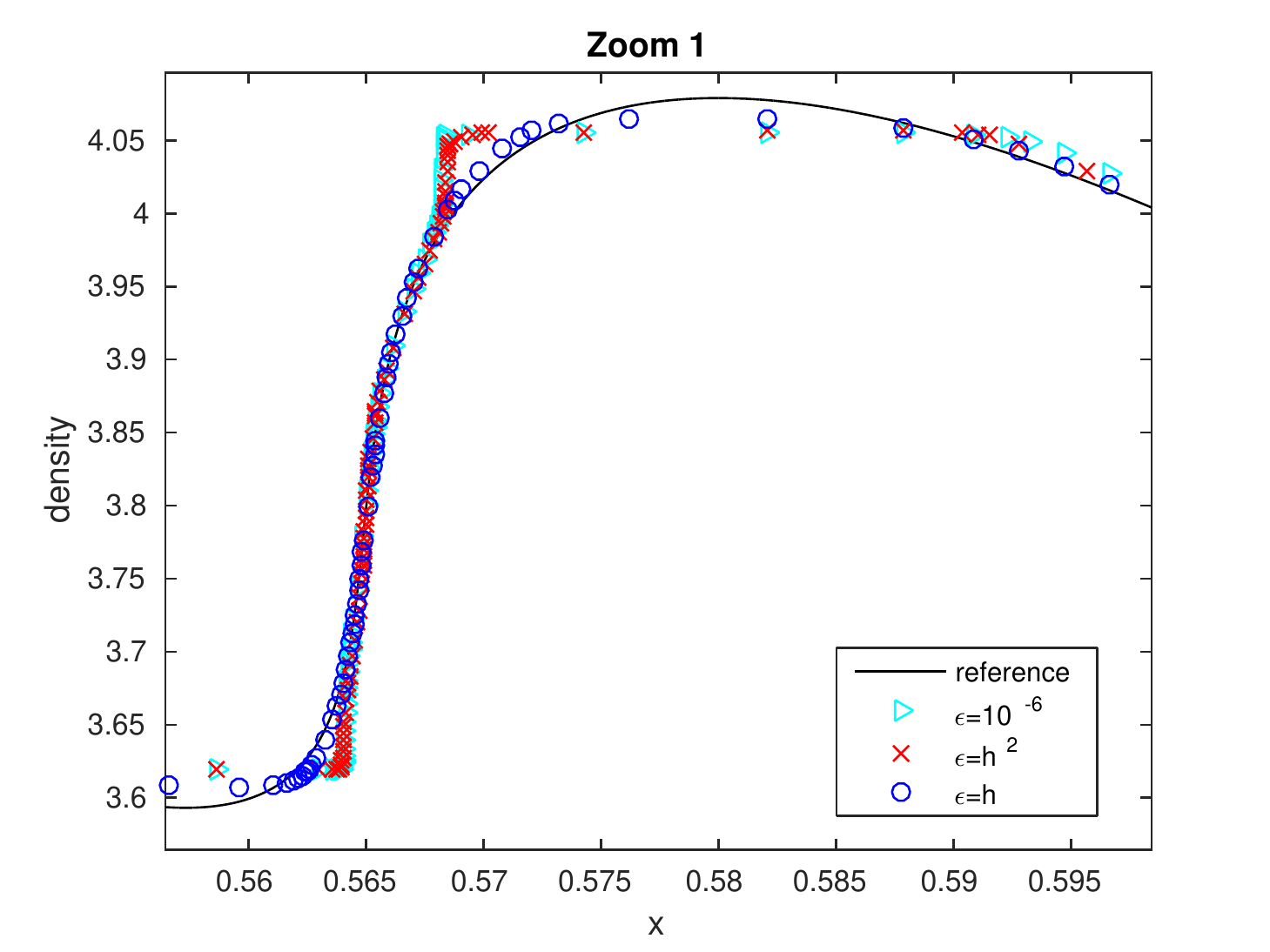}
\hfill
\includegraphics[width=0.3\textwidth]{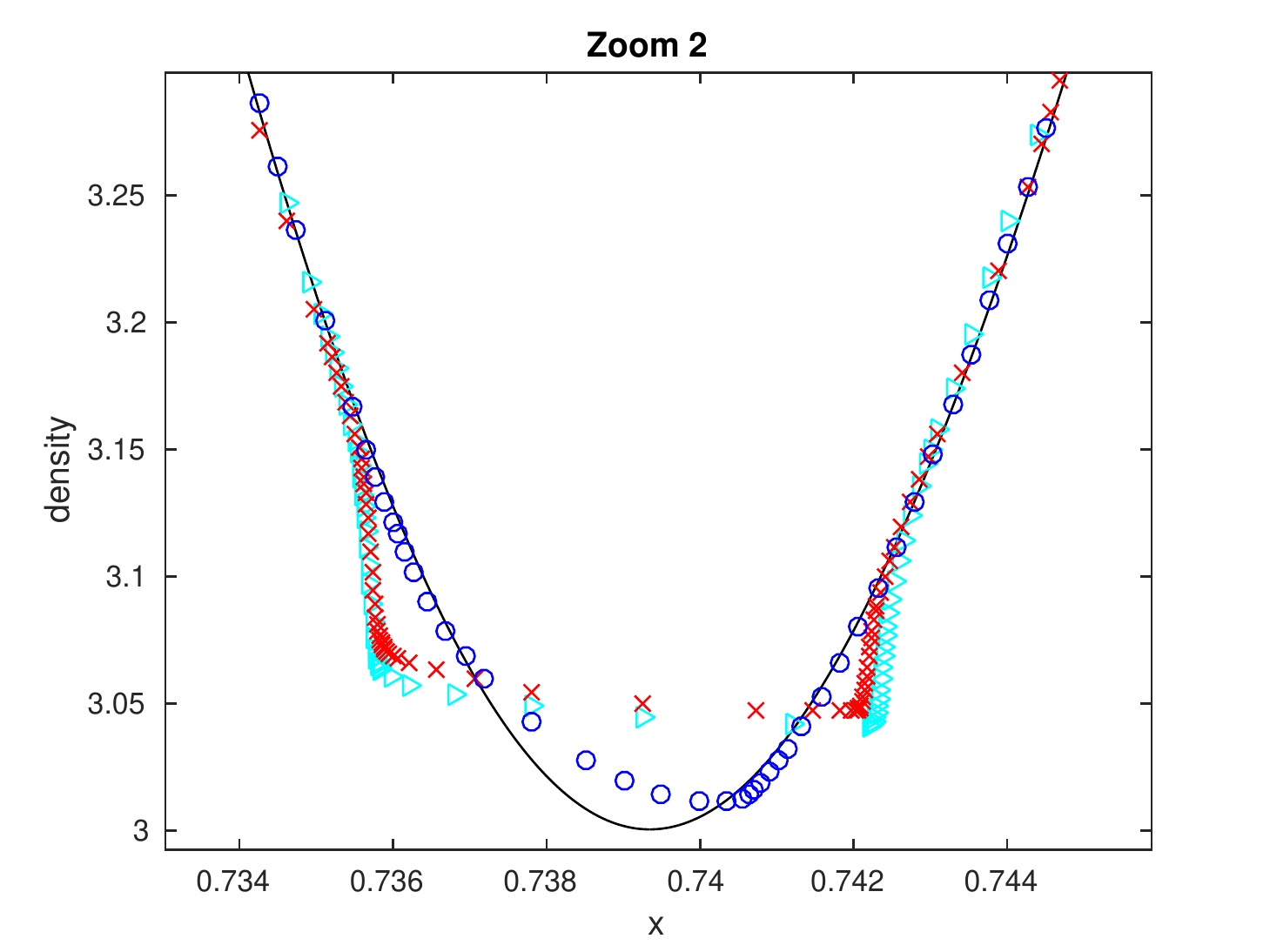}
\hfill
\includegraphics[width=0.3\textwidth]{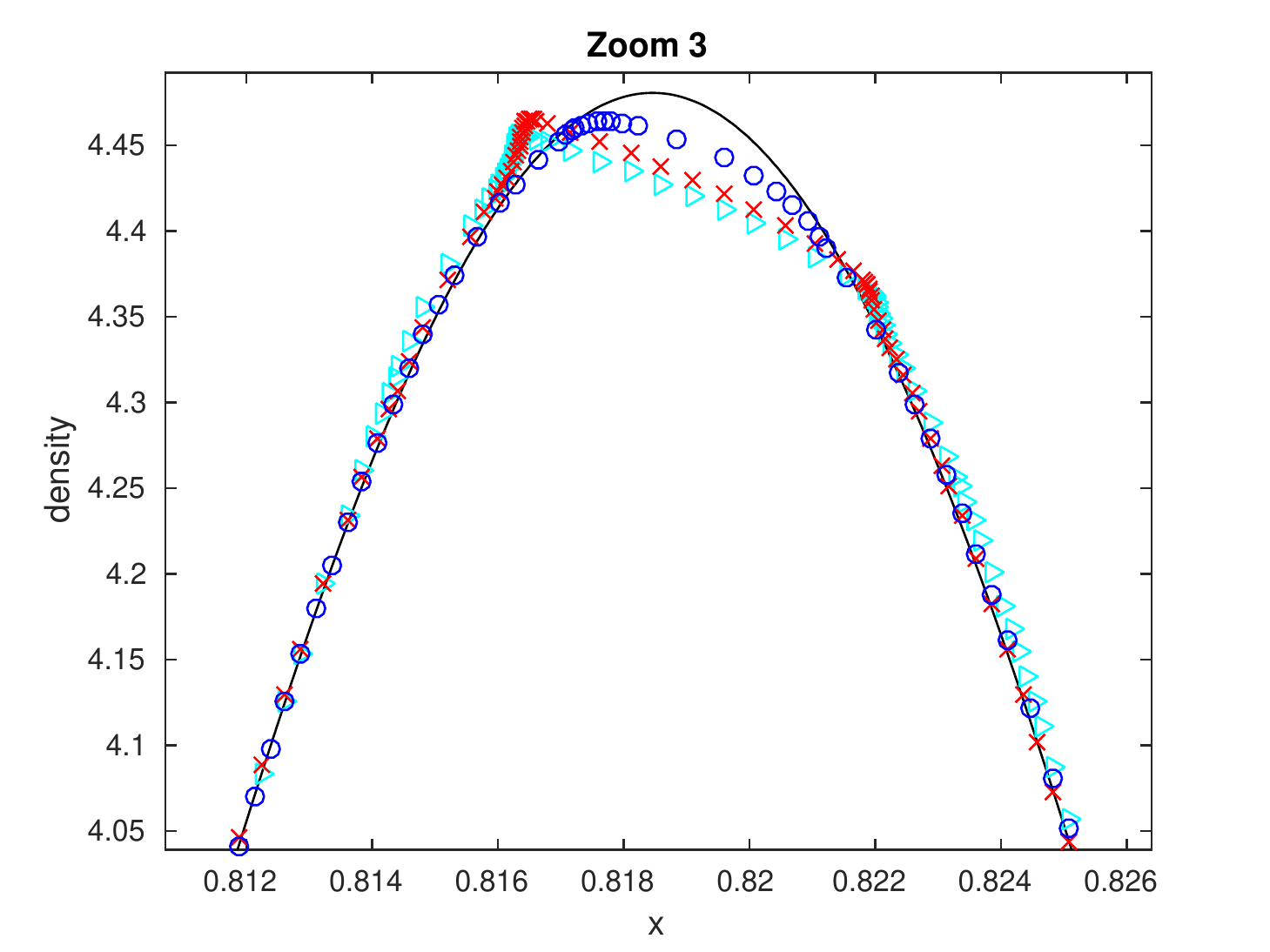}
\end{center}
\caption{Shock-acoustic interaction test. Top-left: efficiency diagram.
Top-right: plot of the density (reference solution). Bottom: comparison of the numerical solutions obtained with different choices of $\epsilon$ in the areas marked with dashed rectangles in the top-right panel.
}
\label{fig:adap:shuosher}
\end{figure}

Finally, we consider the classic problem from \cite{SO88} of the interaction of a shock with a standing acoustic wave. The conservation law is the one-dimensional Euler
equations with $\gamma=1.4$ and the initial data is
\[ 
(\rho,v,p)=\begin{cases}
  (3.857143,2.629369,10.333333),  & x\in[0,0.25],\\
  (1.0 + 0.2\sin(16\pi x), 0.0, 1.0),  & x\in (0.25,1.0].
\end{cases}
\]
The evolution was computed up to $t=0.2$ and the results are presented in Figure \ref{fig:adap:shuosher}. As the right moving shock impinges in the stationary wave, a very complicated smooth structure emerges and then gives rise to small shocks and rarefactions (top right). The evolution was computed with uniform grids and also with the adaptive algorithm, using coarse grids of $32,64,128$ cells, with respectively $6,8,10$ refinement levels. The plot of the error versus the average number of cells employed in each run shows the effectiveness of the h-adaptive procedure and the considerably lower errors obtained when $\epsilon=h$ is employed in the reconstruction procedure (top-left). Three significant portions of the numerical solutions obtained with the h-adaptive algorithm ($128$ coarse cells and $10$ levels) are shown in the lower part of the figure. It is evident that the choice $\epsilon=h$ gives the best results in the approximation of both the high frequency waves and of the little shocks on the left, in particular avoiding the formation of spurious waves, which occour when using $\epsilon=10^{-6}$ or $\epsilon=h^2$.

\paragraph{Third order well-balancing}
We now consider the finite volume discretization of the shallow water equations. It is well known that a second order accurate scheme which is well-balanced for the lake at rest steady state can be constructed with the classical hydrostatic reconstruction technique of \cite{Audusse:2004}. In this scheme, the cell average of the source term is computed as
\[
S_1 = \tfrac12
\left(h_{j-1/2}+h_{j+1/2}\right)
\left(z_{j+1/2}-z_{j-1/2}\right)
\]
where $h_{j\pm1/2}$ and $z_{j\pm1/2}$ are the point values for the water height and the bottom elevation at the boundaries of the $j$-th cell that were employed in the computation of the numerical fluxes. 
A third order accurate scheme may be obtained replacing the original linear, slope-limited, reconstruction with an higher order non-oscillatory one and replacing the original second-order accurate quadrature of the source term with a third-order accurate one. A simple way to obtain such a formula is to apply the Richardson extrapolation technique to $S_1$, as proposed in \cite{NatvigEtAl}. Namely, denoting with $S_2$ the same quadrature rule applied with two sub-intervals, we have that the combination $(4S_2-S_1)/3$ is a fourth-order accurate quadrature for the source term. However, since
\[
S_2 = \tfrac12
\left(h_{j-1/2}+h_{j}\right)
\left(z_{j}-z_{j-1/2}\right)
+ \tfrac12
\left(h_{j}+h_{j+1/2}\right)
\left(z_{j+1/2}-z_{j}\right)
\]
the point-value reconstructions at the cell center must be available as well.

\begin{table}
\begin{center}
\begin{tabular}{l|rrrr}
& $N=100$ & $N=200$& $N=400$& $N=800$\\
\hline
WENO3+P2 & \sn{1.63}{-2} & \sn{1.09}{-2} & \sn{1.51}{-2} & \sn{1.06}{-2}\\
CWENO3 & \sn{6.29}{-4} & \sn{1.27}{-4} & \sn{3.00}{-5} &\sn{6.88}{-6}
\end{tabular}
\end{center}
\caption{Test on the violation of the passivity constraint. Conservation errors of numerical schemes modified by truncating to $0$ all negative water height values.}
\label{tab:cons_err}
\end{table}

As mentioned at the beginning of \S\ref{sec:cweno}, 
WENO3 cannot be applied to compute third order accurate reconstruction at cell centres $h_j$ and $z_j$. Since these values are needed in the well-balanced quadratures, one could use $h_j=\ca{h}_j$ and $z_j=\ca{z}_j$ instead of the third order accurate value, but this would reduce the convergence order to 2 as it can be easily checked on the smooth flow suggested in \cite{XingShu:2005:WBSWEfd}. 
Alternatively, one may reconstruct the point values $h_j$ and $z_j$ by evaluating the central optimal polynomial at the cell center.
Let's call this reconstruction WENO3+P2 and compare it with CWENO3.
Neither CWENO3 nor WENO3+P2 give rise to positive schemes, but the latter causes more oscillations, expecially close to wet-dry fronts.
In order to show the differences, we consider and compare the positive schemes obtained by artificially setting $h_i^{n+1}=0$ whenever $h_i^{n+1}<0$ after each timestep. These schemes are of course not conservative, but their conservation error accumulates over time and gives a measure of the violation of the positivity constraint.
Table \ref{tab:cons_err} shows the conservation errors registered for the simulation of a pond with bottom $z(x)=2x^2$, initial water level $H(x)=z(x)+h(x)=\max( 1+0.4x, z(x) )$ and $q(x)=0$. At time $t=4.0$ the water surface has inverted its movement four times, each time giving rise to spurious waves at the dry points. The scheme based on WENO3+P2 does not show any convergence of the conservation error, indicating that the positivity constraint violations are much severe than those caused by the CWENO3 scheme, which gives rise to a conservation error that decays as $N^{-2}$ with the increase of the number of points. The results in the Table \ref{tab:cons_err} refer to the case of quasi-uniform meshes, but we point out that similar ones were observed on uniform ones.

\subsection{Two-dimensional tests}
A generalisation of the CWENO reconstruction to locally-refined quad-tree grids was described in \cite{CRS}. Here we present a comparison of the different choices for $\epsilon$ in this setting.

The reconstruction presented in \cite{CRS} is a truly two-dimensional reconstruction that generalises the one of \cite{LPR:2001} to non-uniform meshes. In brief, for the cell $\Omega_j$ one considers as ``optimal'' second order polynomial $P_j^{\text{OPT}}(x,y)$ that interpolates the cell average in $\Omega_j$ exactly and all the cell averages in the neighbours in a least-squares sense (this approach is of course needed since in a two-dimensional h-adapted grid the number of neighbours is variable). Additionally the reconstruction considers four ``directional planes'', i.e. first order polynomials that fit exactly the cell average in $\Omega_j$ and, in a least squares sense, the cell averages of the neighbours in the north-east, north-west, south-west and south-east sector. For the precise definition of the neighbourhoods and stencils, see \cite{CRS}. Then a second order polynomial $P_j^{\text{C}}(x,y)$ is defined by requiring that 
\[
P_j^{\text{OPT}}(x,y)=
\tfrac12P_j^{\text{C}}(x,y) 
+\tfrac18P_j^{\text{NE}}(x,y) 
+\tfrac18P_j^{\text{NW}}(x,y) 
+\tfrac18P_j^{\text{SW}}(x,y) 
+\tfrac18P_j^{\text{SE}}(x,y) ,
\]
non-linear weights are computed from the optimal weights $C_\text{C}=\tfrac12$ and $C_\text{NE}=C_\text{NW}=C_\text{SW}=C_\text{SE}=\tfrac18$ with the help of the Jiang-Shu indicators and a reconstruction polynomial is defined by
\[
P_j^{\text{REC}}(x,y)=
\omega_{\text{C}} P_j^{\text{C}}(x,y) 
+\omega_{\text{NE}} P_j^{\text{NE}}(x,y) 
+\omega_{\text{NW}} P_j^{\text{NW}}(x,y) 
+\omega_{\text{SW}} P_j^{\text{SW}}(x,y) 
+\omega_{\text{SE}} P_j^{\text{SE}}(x,y).
\]
This latter polynomial is uniformly third order accurate in the cell (for smooth data) and it can be evaluated where needed. For non-smooth data, the procedure of computing the nonlinear weights from the linear one by using the regularity indicators, ensures that data from non-smooth sub-stencils are not relevant for the coefficients of $P_j^{\text{REC}}$ and the reconstruction is non-oscillatory. For more tests, see \cite{CRS}.

In this paper we want to test the influence of the choice of $\epsilon$ in the accuracy of the reconstruction for smooth data. We remark that the role of $\epsilon$ is quite important in h-adapted meshes, since the grid includes cells of very different size and it is thus quite difficult to choose a fixed value of $\epsilon$ that works well on every cell. 

\begin{table}
\begin{center}
\includegraphics[width=0.5\textwidth]{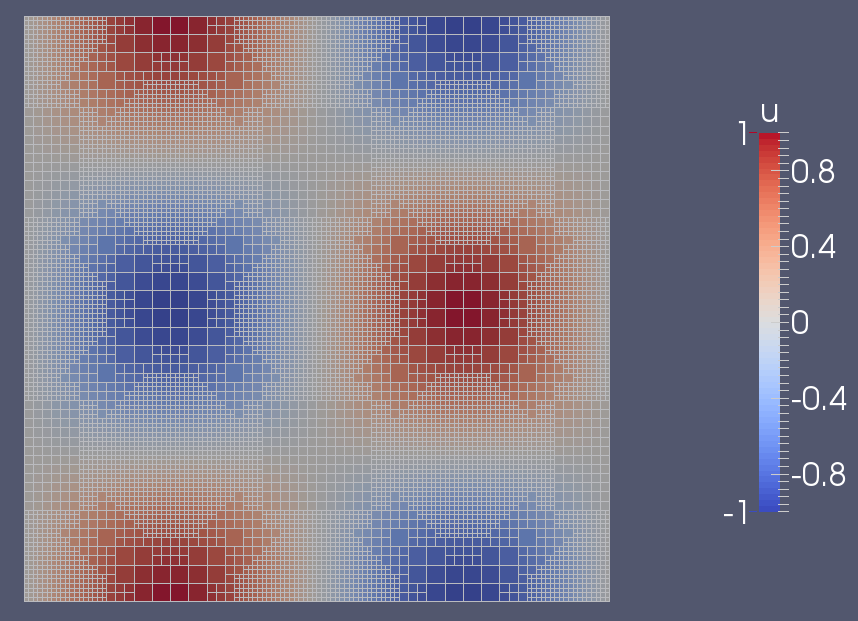}

\begin{tabular}{|lr|ll|ll|ll|}
\hline
&$N$ & \multicolumn{2}{c|}{$\epsilon(h)=10^{-6}$} & \multicolumn{2}{c|}{$\epsilon(h)=h^2$}& \multicolumn{2}{c|}{$\epsilon(h)=h$}\\
&& error & rate & error & rate& error & rate\\
\hline 
$\mathcal{G}_0$ &  1.03e4 &  1.82e-02 &       &  1.30e-02 &       &  2.94e-03 &      \\
$\mathcal{G}_1$ &  4.13e4 &  4.59e-03 &  1.98 &  1.93e-03 &  2.75 &  2.03e-04 &  3.86 \\
$\mathcal{G}_2$ &  1.65e5 &  1.10e-03 &  2.06 &  2.16e-04 &  3.16 &  2.31e-05 &  3.13\\
$\mathcal{G}_3$ &  6.62e5 &  1.20e-04 &  3.19 &  2.04e-05 &  3.40 &  2.91e-06 &  2.99 \\
$\mathcal{G}_4$ &  2.65e6 &  4.97e-06 &  4.59 &  2.08e-06 &  3.30 &  3.65e-07 &  3.00\\
\hline
\end{tabular}

\end{center}

\caption{Maximum norm errors for the two-dimensional CWENO reconstruction on h-adapted grids. The error decay rates are computed with respect to $\sqrt{N}$, where $N$ is the number of cells in the grid. (Due to the choice of grids $\mathcal{G}_k$, $N$ scales as the inverse of the size of the smallest cell in the grid). The grid $\mathcal{G}_0$ and the function $u$ is depicted on top.}
\label{tab:CWENO2d}
\end{table}

To this end we consider the function $u(x,y)=\sin(2\pi x)\cos(2\pi y)$ on the unit square with periodic boundary conditions on a uniform coarse grid. The grid is locally refined with a simple gradient indicator, obtaining the locally adapted grid $\mathcal{G}_0$ depicted in Table \ref{tab:CWENO2d}. Finer grids $\mathcal{G}_k$ are then obtained from this one by splitting each cell of $\mathcal{G}_0$ into $4^k$ equal parts. For each grid, the cell averages of $u$ are computed with a fifth-order gaussian quadrature rule and then the reconstruction of boundary extrapolated data is performed using the two-dimensional CWENO reconstruction. The reconstruction polynomials are computed in each cell and evaluated at the quadrature nodes of the third-order gaussian quadrature, which would be the quadrature of choice for computing fluxes of a finite-volume scheme on such grids. These values are then compared with the exact values of the function $u$ and the maximum-norm error is reported in Table \ref{tab:CWENO2d}.

We observe that the choice $\epsilon(h)=h$ gives the best results both in terms of error decay rates and in terms of absolute values of the error. In fact, the last column of Table \ref{tab:CWENO2d} consistently records lower errors than the other two and has a more regular behaviour of the decay rates.

\section{Conclusions}\label{sec:conclusions}
\cite{Arandiga} and \cite{Kolb2014} discussed the advantage of choosing the parameter $\epsilon$ in WENO and, respectively, CWENO boundary value reconstruction procedures.
In this work we extended their results to non-uniform meshes and concentrate on finite volume schemes, whereas \cite{Arandiga} worked in the finite differences approach.

Our work shows that, also for a non-uniform grid, choosing $\epsilon$ as a function of the local mesh size (i.e. $\epsilon=\epsilon(h_j)$ for the reconstruction of $u^+_{j-1/2}$ and $u^-_{j+1/2}$) allows one to recover the optimal error of convergence even close to local extrema of the function being reconstructed and in general provides a much more regular pattern of error decay, which is a quite important feature for the good performance of error indicators.

We compared the choices $\epsilon=h$ and $\epsilon=h^2$, showing that both can yield the aforementioned improvements and that the former is slightly better on smooth tests, while the latter gives a slightly lower total variation increase in discontinuous cases.

Furthermore, we compared the performance of the different choices for $\epsilon$ in an h-adaptive context, with third order scheme that employs the CWENO3 reconstruction and, at each timestep, locally refines and coarsens the grid according to the numerical entropy production error indicator. We observed that reconstructing in each cell $\Omega_j$ with $\epsilon_j=h_j$ (where $h_j$ here is the local cell size) yields the best results also in the shocked cases.
The paper is completed by a test on the performance of the CWENO3 reconstruction for the shallow water equation with dry states.

Finally, we point out that $\epsilon_j=h_j$ has been already successfully employed in the h-adaptive scheme for conservation laws in two space dimensions based on a two-dimensional generalisation of the CWENO3 reconstruction on locally adapted meshes presented in \cite{CRS} and that more testing on the comparison of different choices for $\epsilon$ is presented here. Since that reconstruction is based on local approximation and not on interpolation, the analysis of that more general situation goes beyond the scope of this paper, but constitutes an interesting line of research.



\section*{Conflict of Interest}
The authors declare that they have no conflict of interest.

\end{document}